\documentclass[11pt]{amsart}
\usepackage{amssymb, amstext, amscd, amsmath, dsfont, amsthm, biblatex, accents}
\usepackage{commath, tabu, xcolor}
\usepackage{tikz}
\usetikzlibrary{shapes,snakes,calc,arrows}
\usetikzlibrary{decorations.pathreplacing,shapes.geometric}
\usetikzlibrary{calc,positioning}

\newcommand{\bA}{{\mathbb{A}}}

\newcommand{\bC}{{\mathbb{C}}}

\newcommand{\bN}{{\mathbb{N}}}

  \newcommand{\A}{{\mathcal{A}}}

  \newcommand{\E}{{\mathcal{E}}}

\renewcommand{\L}{{\mathcal{L}}}

\renewcommand{\P}{{\mathcal{P}}}

  \newcommand{\X}{{\mathcal{X}}}
  \newcommand{\Y}{{\mathcal{Y}}}

\renewcommand{\a}{\alpha}

\newcommand{\ol}{\overline}

\renewcommand{\phi}{\varphi}
\newcommand{\upchi}{{\raise.35ex\hbox{$\chi$}}}

\newcommand{\osc}[1]{\accentset{\circ}{#1}}

\newcounter{exerci}[section]
\renewcommand{\theexerci}{\arabic{exerci}.}

\title{How Free-Free-Boolean Independence Arises in Bi-Free Probability}
\author{Daniel Pepper}
\address{Department of Mathematics and Statistics, York University, 4700 Keele Street, Toronto, Ontario, M3J 1P3, Canada}
\email{daniel.pepper2@gmail.com}
\date{\today}
\thanks{This research was completed with the support of NSERC (Canada) grant RGPIN-2024-06269.}

\newtheorem{theorem}{Theorem}[section]
\newtheorem{lemma}[theorem]{Lemma}
\newtheorem{proposition}[theorem]{Proposition}

\theoremstyle{definition}
\newtheorem{definition}[theorem]{Definition}
\newtheorem{example}[theorem]{Example}
\newtheorem{remark}[theorem]{Remark}
\newtheorem{notation}[theorem]{Notation}
\newtheorem{remandnot}[theorem]{Remark and Notation}

\bibliography{ref}

\begin{document}
\thispagestyle{empty}

\begin{abstract}
This work concerns notions of multi-algebra independence introduced by Liu and how they can be studied in the context of bi-free probability. In particular, we show how the free-free-Boolean independence for triples of algebras can be embedded into and therefore studied from a lens of bi-free probability. It is also shown how its cumulants can be constructed from the bi-free cumulants.
\end{abstract}

\maketitle

\section{Introduction}

The advent of free probability in the 1980's, owing its existence to Voiculescu, brought with it a fertile ground for many rich inquiries into the depths of non-commutative probability. One such question that naturally arose was `\textit{What other kinds of non-commutative probability are there and how are they related to each other?}'  It is an extended branch of this question that this work concerns. 

A hallmark of probability is in its notion of independence. In non-commutative probability theory, independence shows how joint moments in independent algebras is built from the individual algebras. Speicher showed in \cite{Spe97} that there are only five types of this independence. Classical independence, free independence, and Boolean independence are the three symmetric types, while there are two antisymmetric independences called monotone and anti-monotone independence. Each of these can be seen via an action of operators on a corresponding product space. For example, the classical independence arises from an action on a tensor product space, while the free independence arises from an action on a free product space.

Voiculescu originally considered only one-sided actions on the free product in his development of free probability, but later he noticed that he could simultaneously consider both left and right sided actions, giving birth to bi-free independence for pairs of algebras in \cite{Voi14}. He noticed that both classical independence and free independence could be studied through bi-free probability, and conjectured that Boolean independence could also be modeled therein. This conjecture was answered affirmatively by Skoufranis in \cite{Sko14}. Skoufranis also showed that the monotone independence for a pair of algebras could be modeled in a similar fashion. 

The success of bi-free probability opened up questions of what other types of independences exist for tuples of algebras. In \cite{Liu19}, Liu defined several such notions, with emphasis on a free-Boolean independence for pairs of algebras, and a free-free-Boolean independence for triples of algebras in \cite{liu17}. Voiculescu's question then extends to this setting. Can these multi-algebra independences be studied through bi-free probability as well? This work answers this question in the affirmative for the free-free-Boolean (and consequently for the free-Boolean) independence with amalgamation over an algebra $B$. We extend the ideas of Skoufranis in \cite{Sko14} to see how. 

Apart from the introduction, this paper is split into five sections. Section 2 introduces the basic structures of non-commutative probability with amalgamation. Abstractly, these take the form of $B$-$B$-non-commutative probability spaces which act on $B$-$B$-bimodules with specified $B$-projections. Free product modules are recalled to make sense out of bi-free independence. We define the Boolean projections on the free product, which form a key part of the definition of Liu's free-free-Boolean independence. 

Section 3 discusses a very useful calculus arising from the incidence algebra on the so-called bi-non-crossing partitions. The LR diagram calculus from \cite{CNS15} is called upon to handle vector components of actions, and combined with our Boolean projections. These combinatorial tools are a staple of many non-commutative probability techniques.

Section 4 serves as a very brief introduction to free-free-Boolean independence with amalgamation. Originally defined in the non-amalgamated setting in \cite{liu17}, we extend the definition to the amalgamated case and mention a couple of quick facts. 

In section 5 we tackle the main result of this paper, which shows that free-free-Boolean independent families of triples of algebras can be embedded into, and therefore studied through, bi-free probability. To do so, a bi-free family with a few extra moment conditions are leveraged to house free-free-Boolean independent families. An explicit representation is discussed, and we demonstrate that the bi-free cumulants can be used to define free-free-Boolean cumulants.

Section 6 concludes this work with a brief comment on Liu's other independences.

\section{Preliminaries on $B$-valued probability and Bi-Free Independence}

\subsection{Non-commutative Probability Spaces and Bimodules}

Throughout this work, we denote by $B$, a unital algebra over $\bC$. In $B$-valued probability theory, we concern ourselves with algebras of random variables and $B$-valued distributions. More formally, we can define a $B$-valued probability space as follows.

\begin{definition}
    A $B$-valued probability space is a pair $(\A, E)$ where $\A$ is an algebra over $B$ and $E : \A \to B$ is a linear map such that
    \[ E(b_1 a b_2) = b_1 E(a) b_2 \]
    for all $a \in \A$ and $b_1, b_2 \in B$.
\end{definition}

As the notation might suggest, we often call elements of $\A$ \textit{non-commutative random variables} and $E$ an \textit{expectation} onto the algebra $B$.

We can consider elements of $\A$ as operators, most simply, if we think of these elements as acting on $\A$ by left (or right) multiplication. This turns $\A$ into a $B$-$B$-bimodule, and embeds $\A$ into the space of linear operators on $\A$. The expectation returns to us in the form of a projection onto $B$.

\begin{definition}
    A $B$-$B$-bimodule with a specified $B$-projection is a triple $(\X, \osc{\X}, p)$ where $\X$ is a direct sum of $B$-$B$-bimodules $\X = B \oplus \osc{\X}$ and $p : \X \to B$ is the projection (with kernel $\osc{\X}$) 
    \[ p(b \oplus \eta) = b.\]
Let $\L(\X)$ denote the algebra of $\bC$-linear maps on $\X$. We have an \textit{expectation}, $E_{\L(\X)} : \L(\X) \to B$ defined by 
    \[ E_{\L(\X)}(a) = p(a(1_B \oplus 0)) = pa1_B.\]
\end{definition}

One can check that starting with a $B$-$B$-bimodule with specified projection $(\X, \osc{\X}, p)$, we get for free a $B$-valued probability space $(\L(\X), E_{\L(\X)})$. As alluded to above, the same is true in reverse; starting with a $B$-valued probability space $(\A, E)$, we can create a $B$-$B$-bimodule with specified projection, $(\A, \ker(E), E)$. This correspondence between $B$-valued probability and $B$-$B$-bimodules lies at the heart of many of our investigations that follow.

$\ $

Now, in order to talk about pairs of faces and bi-free probability, we need to be able to discuss left and right acting operators. These behave slightly differently when we consider the algebra $B$ over which we are amalgamating. For example, if $b\in B$, let $L_b$ and $R_b$ denote the action on a $B$-$B$-bimodule $\X$ by left and right multiplication by $b \in B$ respectively, we can consider the left algebra of $\L(\X)$ to be 
\[ \L_\ell(\X) := \{ T \in \L(\X) \ | \ TR_b = R_bT \} \]
and the right algebra of $L(\X)$ to be
\[ \L_r(\X) := \{ T \in \L(X) \ | \ TL_b = L_bT \}.\]
In other words, left-acting operators should commute with right multiplication by $B$, and right acting operators should commute with left-multiplication by $B$. Motivated by this, we look at the following abstraction, that of a $B$-$B$-non-commutative probability space.

\begin{definition}
    A \textit{$B$-$B$-non-commutative probability space} is a triple $(\A, E_\A, \varepsilon)$ where $\A$ is a unital algebra, $\varepsilon : B \otimes B^{\text{op}} \to \A$ is a unital homomorphism such that $\varepsilon|_{B \otimes 1_B}$ and $\varepsilon|_{1_B \otimes B}$ are injective, and $E_\A : \A \to B$ is a linear map such that
    \[ E_\A( \varepsilon(b_1 \otimes b_2) T) = b_1 E_\A(T) b_2\]
    for all $b_1, b_2 \in B$ and $T \in \A$, and 
    \[ E_\A( T\varepsilon(b \otimes 1_B)) = E_\A(T\varepsilon(1_B \otimes b))\]
    for all $b \in B$ and $T \in \A$.

    We define the \textit{left} and \textit{right algebras} of $\A$ respectively by
    \[ \A_\ell := \{ T \in \A \  | \ T \varepsilon(1_B \otimes b) = \varepsilon(1_B \otimes b) T \text{ for all } b \in B\}\]
    and 
    \[ \A_r := \{ T \in \A \ | \  T \varepsilon(b \otimes 1_B) = \varepsilon(b \otimes 1_B) T \text{ for all } b \in B\}.\]
    Often, we use $L_b$ to denote $\varepsilon(b \otimes 1_B)$ and $R_b$ to denote $\varepsilon(1_B \otimes b)$ for simplicity. 
\end{definition}

Previously we saw a connection between $B$-valued probability spaces and $B$-$B$-bimodules with specified projections. This connection can be refined in the case of $B$-$B$-non-commutative probability spaces, with a bit more attention to detail. 

\begin{theorem}[\cite{CNS15}, Theorem 3.2.4] \label{faithfulbimodule} 
    Let $(\A, E_\A, \varepsilon)$ be a $B$-$B$-non-commutative probability space. Then there is a $B$-$B$-bimodule with a specified $B$-projection $(\X, \osc{\X}, p)$ and a unital homomorphism $\theta : \A \to \L(\X)$ such that $\theta(L_{b_1}R_{b_2}) = L_{b_1}R_{b_2}$,
    \[ \theta(\A_\ell) \subseteq \L_\ell(\X), \quad \theta(\A_r) \subseteq L_r(\X), \quad \text{ and } \quad E_{\L(\X)}(\theta(T))= E_\A(T) \]
    for all $b_1, b_2 \in B$ and $T \in \A$. 
\end{theorem}

Because the details are necessary for later use, we mention the basic construction that verifies the above result.

Given $(\A, E_\A, \varepsilon)$, let 
\[ \X = B \oplus \left( \ker(E_\A)/\text{span}\{TL_b - TR_b | T \in \A, b \in B \} \right) .\]
From here, a unital homomorphism $\theta : \A \to \L(\X)$ is defined with the required properties by way of 
\[ \theta(T)(b) = E_\A(TL_b) \oplus q(TL_b -L_{E_\A(TL_b)})\]
and
\[ \theta(T)(q(A)) = E_\A(TA) \oplus q(TA -L_{E_\A(TA)})\]
for each $A \in \A$ with $E_\A(A) = 0$, $T \in \A$, and $b \in B$. Here $q$ is the quotient map 
\[ q : \ker E_\A \to \ker(E_\A)/\text{span}\{TL_b - TR_b | T \in \A, b \in B \}. \]

The $B$-$B$-bimodule structure on $\X$ is then defined by 
\[ b \cdot \xi = \theta(L_b)(\xi) \quad \text{ and } \quad \xi \cdot b = \theta(R_b)(\xi)\]
for all $b \in B$ and $\xi \in \X$.

\begin{remandnot} \label{bbbimoduleconstruction}
With notations as above, we can form a new $B$-$B$-bimodule with specified $B$-projection by taking a direct sum. In particular, let
\[ \Y = \X \oplus \X, \quad \text{ and } \quad \osc{\Y} = \osc{\X} \oplus \X.\]
So that $p_\Y(\xi_1 \oplus \xi_2) = p_\X(\xi_1) \oplus 0$ for all $\xi_1, \xi_2 \in \X$.

\end{remandnot}

\subsection{Free Product Modules and Bi-Free Independence}

With our basic structures in hand, we discuss the main reason why we call this \textit{probability} by highlighting the notion of \textit{independence} between algebras or random variables that we care about. Independence between algebras arises naturally from their actions on a corresponding product space. Let's consider the free product.

\begin{definition}

Let $\{(\X_k, \osc{\X}_k, p_k)\}_{k \in K}$ be a family of $B$-$B$-bimodules with specified $B$-projections. The \textit{reduced free product of $\{(\X_k, \osc{\X}_k, p_k)\}_{k \in K}$ with amalgamation over $B$} is defined to be the $B$-$B$-bimodule with specified $B$-projection $(\X, \osc{\X}, p)$, where $\osc{\X}$ is the $B$-$B$-bimodule defined by 
\[ \osc{\X} = \bigoplus_{n \geq 1} \bigoplus_{k_1 \neq k_2 \neq \cdots \neq k_n} \osc{\X}_{k_1} \otimes_B \cdots \otimes_B \osc{\X}_{k_n}.\]

For each $k \in K$, we collect the submodule generated by all those tensor words which do not start with elements of $\osc{\X}_k$, and denote it by
\[ \X_{\ell}(k) = B \oplus \bigoplus_{n \geq 1} \bigoplus_{\substack{ k_1 \neq k_2 \neq \cdots \neq k_n \\ k_1 \neq k}} \osc{\X}_{k_1} \otimes_B \cdots \otimes_B \osc{\X}_{k_n}, \]
and let $V_k$ be the natural isomorphism of $B$-$B$-bimodules
\[ V_k : \X \to \X_k \otimes_B \X_{\ell}(k). \]

Similarly, we can denote the submodule generated by all those tensor words which do not end with the elements of $\osc{\X}_k$. We denote this by
\[ \X_r(k) = B \oplus \bigoplus_{n \geq 1} \bigoplus_{\substack{k_1 \neq k_2 \neq \cdots \neq k_n \\ k_n \neq k}} \osc{\X}_{k_1} \otimes_B \cdots \otimes_B \osc{\X}_{k_n}. \]
Let also $W_k$ be the natural isomorphism of $B$-$B$-bimodules
\[ W_k : \X \to \X_r(k) \otimes_B \X_k\]

For each $k \in K$, denote by $\lambda_k : \L(\X_k) \to \L(\X)$ the unital homomorphism called the left regular representation of $\L(\X_k)$ on $\X$, defined by 
\[ \lambda_k(a) = V_k^{-1}(a \otimes \text{Id})V_k,\]
and $\rho_k : \L(\X_k) \to \L(\X)$ the unital homormorphism called the right regular representation of $\L(\X_k)$ on $\X$ defined by 
\[ \rho_k(a) = W_k^{-1}(\text{Id} \otimes a)W_k.\]
\end{definition}

Then, bi-free independence for pairs of $B$-faces can be defined as follows.

\begin{definition} 
    Let $(\A, E_\A, \varepsilon)$ be a $B$-$B$-non-commutative probability space. A \textit{pair of $B$-faces of $\A$} is a pair $(C,D)$ of unital subalgebras of $\A$ such that 
    \[ \varepsilon(B \otimes 1_B) \subseteq C \subseteq \A_\ell \quad \text{ and } \quad \varepsilon(1_B \otimes B) \subseteq D \subseteq A_r.\]
    A family $\{(C_k, D_k)\}_{k \in K}$ of pairs of $B$-faces of $\A$ is said to be \textit{bi-free over $B$} if there exist $B$-$B$-bimodules with specified $B$-projections $\{(\X_k, \osc{\X}_k, p_k)\}_{k \in K}$ and unital homomorphisms $\ell_k: C_k \to \L_\ell(\X_k)$, $r_k : D_k \to \L_r(\X_k)$ such that the joint distribution of $\{(C_k, D_k)\}_{k \in K}$ with respect to $E_\A$ is equal to the joint distribution of the family 
    \[ \{ ((\lambda_k \circ \ell_k)(C_k), (\rho_k \circ r_k)(D_k))\}_{k \in K}\]
    in $(\L(\X), E_{\L(\X)})$ where $\X$ is the reduced free product of the $\X_k$, and $\lambda_k$ and $\rho_k$ are respectively the left and right regular representations.
\end{definition}

\subsection{Boolean Projections on the Free Product}

One can define Boolean independence similarly through an action on a free-product space. But to do so, we must project onto a special submodule.

\begin{definition} \label{boolprojdef}
    Let $\{(\X_k, \osc{\X}_k, p_k)\}_{k \in K}$ be a family of $B$-$B$-bimodules with specified $B$-projections, and let $(\X, \osc{\X}, p)$ be their reduced free product. For each $k \in K$ we define the \textit{$k^{\text{th}}$ Boolean projection} $P_k : \X \to \X$, to be the natural projection onto $B \oplus \osc{\X}_k$.
\end{definition}

We collect a couple of easy facts regarding these Boolean projections. 

\begin{lemma} \label{boolprojfacts}
    With notations as above, fix $k \in K$. If $T \in \L(\X_k)$ then 
        \[ \lambda_k(T)P_k = P_k \lambda_k(T) \quad \text{ and } \quad \rho_k(T) P_k = P_k \rho_k(T),\]
    and
        \[ \lambda_k(T) P_k =  \rho_k(T) P_k.\]
\end{lemma}

\begin{proof}
    The first claim follows easily from the definitions. For the second claim, note that for any $\eta \in \X$, we have $P_k \eta \in \X_k$ so that 
    \[ \lambda_k(T) P_k \eta = T P_k \eta = \rho_k(T) P_k \eta. \qedhere \]
\end{proof}

\section{LR diagrams and bi-free cumulants}

Instead of the abstract approach of looking for products of bimodules to see independence, we can instead look at patterns in moment calculations (that is, expectations of words) between independent algebras. This gets captured neatly in the combinatorial framework of bi-non-crossing partitions which act as a medium through which to define the bi-free cumulants. We will use these to define cumulants for the free-free-Boolean independence later.

\subsection{Bi-non-crossing partitions}

We begin with the combinatorial objects that encode our moment-cumulant calculus. Let $\P(n)$ denote the set of partitions of $\{1, \ldots, n\}$. Then $\P(n)$ is a lattice with the refinement order: say $\pi$ is \textit{finer} than $\sigma$ and write $\pi \leq \sigma$ if each block (set) of $\pi$ is contained in a block (set) of $\sigma$. For $\pi \in \P(n)$ denote by $\sim_\pi$ the equivalence relation on $\{1, \ldots, n\}$ corresponding to $\pi$. That is, $i \sim_\pi j$ if $i$ and $j$ are in the same block of $\pi$. A partition $\pi$ is said to be \textit{non-crossing} if whenever $a_1 < b_1 < a_2 < b_2$ with $a_1 \sim_\pi a_2$ and $b_1 \sim_\pi b_2$, then $a_1 \sim_\pi b_1$ (that is, all of $a_1, a_2, b_1, b_2$ must be in the same block of $\pi$). The set of non-crossing partitions, often denoted by $NC(n)$, is itself a lattice under the refinement order. 

Since we deal with left and right acting operators, we should pay attention to this. To do so, let $\upchi : \{ 1, \ldots, n\} \to \{ \ell, r\}$. We form a permutation associated with $\upchi$ which reorders the elements $\{1, \ldots, n\}$ by taking into account if they should be left or right operators. In particular, order the lefts in the usual way:
\[ \upchi^{-1}(\{\ell\}) = \{ i_1 < \cdots < i_p\},\]
and for the rights, we order them in reverse, as right acting operator multiplication gets reversed
\[ \upchi^{-1}(\{r\}) = \{ i_{p+1} > \cdots > i_n \}.\]
Now define our permutation $s_\chi$ by $s_\chi(j) = i_j$. It is convenient to put a new ordering, $\prec_\chi$ on $\{1, \ldots, n\}$ by $a \prec_\chi b$ if and only if $s_\chi^{-1}(a) < s_\chi^{-1}(b)$. This ordering is that which is induced by $s_\chi$. 

Then $s_\chi$ is an order-isomorphism from $(\{1, \ldots, n\}, <)$ to $(\{1,\ldots, n\}, \prec_\chi)$. So we can use this to transform the non-crossing partitions into the so-called \textit{bi-non-crossing} partitions.

\begin{definition}
    A partition $\pi \in \P(n)$ is said to be \textit{bi-non-crossing} with respect to $\upchi$ if the permutation $s_\chi^{-1} \cdot \pi$ is non-crossing. Denote by $BNC(\upchi)$ the set of bi-non-crossing partitions with respect to $\upchi$. 
\end{definition}

It is a simple observation to note that $s_\chi$ upgrades to a natural order-isomorphism of $\P(n)$, so that $BNC(\upchi)$ is isomorphic to $NC(n)$, granting us that $BNC(\upchi)$ is itself a lattice. 

There is a nice visual representation associated to each of these partitions, a \textit{bi-non-crossing diagram}. Let $\pi \in BNC(\chi)$, and create its bi-non-crossing diagram as follows. Along two parallel vertical dashed lines, place nodes labeled $1$ to $n$ from top to bottom, such that the nodes on the left line correspond to those values $k$ such that $\chi(k) = \ell$ and the nodes on the right correspond to those values of $k$ such that $\chi(k) = r$. Connect nodes together whose labels share a $\pi$ block with lines. This can be done in a way so that lines from different blocks do not cross.

Here is a simple example. Let $\chi : \{1, 2,3,4,5, 6\} \to \{\ell, r\}$ such that $\chi^{-1}(\{ \ell \}) = \{1, 3, 4, 5\}$, $\chi^{-1}(\{r\}) = \{2,6\}$, let $\pi = \{ \{ 1, 2, 5, 6\}, \{3,4\} \}$ and $\sigma = \{\{ 1, 4, 5 , 6\}, \{2, 3 \}\}$. Then $\pi$ is bi-non-crossing with respect to $\chi$, while $\sigma$ is not. This is visually apparent from their diagrams below.

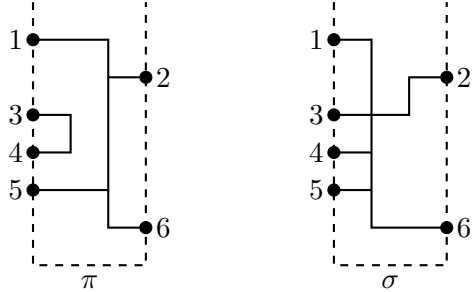
\begin{center}

    \begin{tikzpicture}[baseline]

	\draw[thick, dashed] (0,3.5) -- (0,0) -- (1.5, 0) -- (1.5,3.5);

	\draw[fill=black] (0, 3) circle (0.08);

	\node[left] at (0, 3) {$1$};

	\draw[fill=black] (1.5, 2.5) circle (0.08);

	\node[right] at (1.5, 2.5) {$2$};

	\draw[fill=black] (0, 2) circle (0.08);

	\node[left] at (0, 2) {$3$};

	\draw[fill=black] (0, 1.5) circle (0.08);

	\node[left] at (0, 1.5) {$4$};

	\draw[fill=black] (0, 1) circle (0.08);

	\node[left] at (0, 1) {$5$};

	\draw[fill=black] (1.5, 0.5) circle (0.08);

	\node[right] at (1.5, 0.5) {$6$};

	\draw[thick] (0, 2)  -- (0.5, 2) -- (0.5, 1.5) -- (0, 1.5);

        \draw[thick] (0, 3) -- (1, 3) -- (1, 2.5) -- (1.5, 2.5) ;

        \draw[thick] (1,2.5) -- (1, 1) -- (0, 1) ;

        \draw[thick] (1, 1) -- (1, 0.5) -- (1.5, 0.5);

        \node[below] at (0.75,0) {$\pi$};

        \draw[thick, dashed] (4,3.5) -- (4,0) -- (5.5, 0) -- (5.5,3.5);

	\draw[fill=black] (4, 3) circle (0.08);

	\node[left] at (4, 3) {$1$};

	\draw[fill=black] (5.5, 2.5) circle (0.08);

	\node[right] at (5.5, 2.5) {$2$};

	\draw[fill=black] (4, 2) circle (0.08);

	\node[left] at (4, 2) {$3$};

	\draw[fill=black] (4, 1.5) circle (0.08);

	\node[left] at (4, 1.5) {$4$};

	\draw[fill=black] (4, 1) circle (0.08);

	\node[left] at (4, 1) {$5$};

	\draw[fill=black] (5.5, 0.5) circle (0.08);

	\node[right] at (5.5, 0.5) {$6$};

	\draw[thick] (4, 2)  -- (5, 2) -- (5, 2.5) -- (5.5, 2.5);

        \draw[thick] (4,3) -- (4.5, 3) -- (4.5, 0.5) -- (5.5, 0.5);

        \draw[thick] (4, 1.5) -- (4.5, 1.5);

        \draw[thick] (4, 1) -- (4.5,1);

        \node[below] at (4.75,0) {$\sigma$};

	\end{tikzpicture}

\end{center}

In bi-non-crossing diagrams, we call the vertical lines in each block \textit{spines}.

Since we are dealing with a lattice, we can define an incidence algebra and get a M\"obius function.

\begin{definition}
    The \textit{bi-non-crossing M\"obius function} is the function 
    \[ \mu_{BNC} : \bigcup_{n \geq 1} \bigcup_{\upchi:\{1, \ldots, n\} \to \{ \ell, r\}} BNC(\upchi) \times BNC(\upchi) \to \bC\]
    defined so that $\mu_{BNC}(\pi, \sigma) = 0$ unless $\pi$ is a refinement of $\sigma$, and otherwise
    \[ \sum_{\substack{\tau \in BNC(\upchi) \\ \pi \leq \tau \leq \sigma}}\mu_{BNC}(\tau, \sigma) = \sum_{\substack{\tau \in BNC(\upchi) \\ \pi \leq \tau \leq \sigma}} \mu_{BNC} (\pi, \tau) = \begin{cases} 1 & \text{ if } \pi = \sigma \\ 0 & \text{ otherwise} \end{cases} \ \ .\]
\end{definition}

\subsection{Bi-Multiplicative Functions}

Our interest in these partitions comes from their natural presence in computing moments between bi-freely independent pairs of algebras. To see how, first we recall the definition of the operator-valued bi-multiplicative functions from \cite{CNS15}.

\begin{definition} \label{bimultiplicative}
    Let $(\A, E_\A, \varepsilon)$ be a $B$-$B$-non-commutative probability space. Fix $\upchi : \{1, \ldots, n\} \to \{ \ell, r \}$, $\pi \in BNC(\chi)$, and $Z_1, \ldots, Z_n \in \A$. Define
    \[ E_\pi(Z_1, \ldots, Z_n) \in B \]
    recursively as follows. Let $V$ be the block of $\pi$ with the largest minimum element. Then,
    \begin{itemize}
        \item If $\pi = 1_{\chi}$ then
        \[ E_\pi(Z_1, \ldots, Z_n) = E_\A(Z_1 \cdots Z_n).\]
        
        \item If there is some $k \in \{1, \ldots n-1\}$ such that $V = \{ k+1, \ldots, n\}$, then
        \[ E_\pi (Z_1, \ldots, Z_n) = E_{\pi |_{V^c}}(Z_1, \ldots , Z_k L_{E_\A(Z_{k+1} \cdots Z_n)}).\]
        
        \item Otherwise, $\min(V)$ is adjacent to a spine. Let $W$ be the block of $\pi$ corresponding to the spine adjacent to $\min(V)$. Let $k$ be the smallest element of $W$ such that $k > \min(V)$. Then define 
        \[ E_\pi(Z_1, \ldots, Z_n) = \begin{cases} E_{\pi|_{V^c}}((Z_1, \ldots, Z_{k-1}, L_{E_{\pi |_V}((Z_1,\ldots, Z_n)|_V)}Z_k, \ldots, Z_n)|_{V^c})  &\text{ if } \chi(\min(V)) = \ell \\  E_{\pi|_{V^c}}((Z_1, \ldots, Z_{k-1}, R_{E{\pi |_V}((Z_1,\ldots, Z_n)|_V)}Z_k, \ldots, Z_n) |_{V^c})&\text{ if } \chi(\min(V)) = r \end{cases}.\]
    \end{itemize}
\end{definition}

\begin{definition}
    Let $(\A, E_\A, \varepsilon)$ be a $B$-$B$-noncommutative probability space. The \textit{bi-free operator-valued moment function} 
    \[ \E : \bigcup_{n \geq 1} \bigcup_{\upchi:\{1, \ldots, n\} \to \{\ell, r\}} BNC(\upchi) \times \A_{\chi(1)} \times \cdots \times \A_{\chi(n)} \to B\]
    is defined, for each $\upchi : \{1, \ldots, n\} \to \{ \ell, r \}$, $\pi \in BNC(\upchi)$, and $Z_k \in \A_{\chi(k)}$, by
    \[ \E_\pi(Z_1, \ldots, Z_n) = E_\pi(Z_1, \ldots, Z_n).\]

    Similarly, the \textit{bi-free operator-valued cumulant function} 
    \[ \kappa : \bigcup_{n \geq 1} \bigcup_{\upchi:\{1, \ldots, n\} \to \{\ell, r\}} BNC(\upchi) \times \A_{\chi(1)} \times \cdots \times \A_{\chi(n)} \to B\]
    is defined, for each $\upchi : \{1, \ldots, n\} \to \{ \ell, r \}$, $\pi \in BNC(\upchi)$, and $Z_k \in \A_{\chi(k)}$, by
    \[ \kappa_\pi(Z_1, \ldots, Z_n) = \sum_{\substack{\sigma \in BNC(\chi) \\ \sigma \leq \pi}} \E_\sigma(Z_1, \ldots, Z_n) \mu_{BNC}(\sigma,\pi).\]
\end{definition}

The useful property for us is that the cumulants vanish precisely when their arguments come from mixtures of bi-freely independent families. This is captured by the following result (see Theorems 7.1.4 and 8.1.1 in \cite{CNS15}).

\begin{theorem} \label{combinatorialbifree}
    Let $(\A, E_\A, \varepsilon)$ be a $B$-$B$-non-commutative probability space and let $\{(C_k, D_k)\}_{k \in K}$ be a family of pairs of $B$-faces in $\A$. Then $\{(C_k, D_k)\}_{k \in K}$ are bi-free with amalgamation over $B$ if and only if for all $\chi : \{1, \ldots, n\} \to \{\ell, r\}$, $\epsilon : \{1, \ldots, n\} \to K$, and $Z_k$ such that $Z_k \in C_{\epsilon(k)}$ if $\chi(k) = \ell$ and $Z_k \in D_{\epsilon(k)}$ if $\chi(k) = r$, we have the equality 
    \[ E(Z_1 \cdots Z_n) = \sum_{\substack{\pi \in BNC(\chi) \\ \pi \leq \epsilon}} \sum_{\substack{\sigma \in BNC(\chi)\\ \sigma \leq \pi}} \mu_{BNC(\chi)}(\pi, \sigma) \E_\pi(Z_1, \ldots, Z_n).\]
    Equivalently, $\{(C_k, D_k)\}_{k \in K}$ are bi-free with amalgamation over $B$ if and only if 
    \[ \kappa_{1_\chi}(Z_1, \ldots, Z_n) = 0\]
    whenever $\epsilon$ is not constant.
\end{theorem}

\subsection{LR diagrams}

More generally, we want to keep track of the vector components of the action of words from independent algebras on reduced free product space. So we turn to another, but similar, set of diagrams to help us describe the pieces. 

Throughout this section, in addition to a function $\chi : \{1, \ldots, n\} \to \{\ell, r\}$, let $\epsilon : \{1, \ldots, n\} \to K$ for some fixed index set $K$.

\begin{definition}
    The set $LR(\chi,\epsilon)$ of \textit{shaded $LR$ diagrams} is defined recursively. When $n = 0$, $LR(\chi,\epsilon)$ contains only an empty diagram. If $n > 0$, let $\ol{\chi}(k) = \chi(k-1)$ and $\ol{\epsilon}(k) = \epsilon(k-1)$ for $k \in \{2, \ldots, n\}$. Then, each $D \in LR(\ol{\chi},\ol{\epsilon})$ corresponds to two unique elements of $LR(\chi,\epsilon)$ by the following processes:
    \begin{itemize}
        \item First, add to the top of $D$ a node on the side corresponding to $\chi(1)$, shaded by $\epsilon(1)$.
        \item If a string of shade $\epsilon(1)$ extends from the top of $D$, connect it to the added node. 
        \item Then, choose to either extend a string from the added node to the top of the new diagram or not, and extend any other strings from $D$ to the top of the new diagram.
    \end{itemize}
    We denote by $LR_k(\chi,\epsilon)$, the subset of $LR(\chi(\epsilon))$ consisting of diagrams with exactly $k$ strings that reach the top gap.

    We refer to the vertical portions of strings in the diagrams as \textit{spines} and horizontal segments connecting nodes to spines as \textit{ribs}.
\end{definition}

\begin{example}
As an example, consider $\chi = (\ell, r, \ell)$ and $\epsilon = (', ', '')$, Then $\overline{\chi} = (r, \ell)$ and $\overline{\epsilon} = (', '')$. Consider the diagram $D \in LR(\ol{\chi},\ol{\epsilon})$ formed by extending a spine from the node $1$ to the top gap, and leaving the node $2$ isolated. Then the two diagrams we get from extending $D$ are $D_1, D_2 \in LR(\chi,\epsilon)$, pictured below.

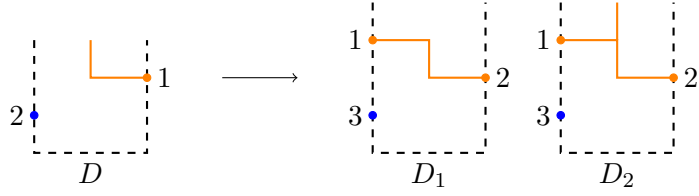
\begin{center}

    \begin{tikzpicture}[baseline]

	\draw[thick, dashed] (0,1.5) -- (0,0) -- (1.5, 0) -- (1.5,1.5);

	\draw[orange, fill=orange] (1.5, 1) circle (0.05);

	\node[right] at (1.5, 1) {$1$};

	\draw[blue, fill=blue] (0, 0.5) circle (0.05);

	\node[left] at (0, 0.5) {$2$};

        \draw[orange, thick] (1.5, 1) -- (0.75, 1) -- (0.75, 1.5);

        \node[below] at (0.75,0) {$D$};

        \draw[->] (2.5, 1) -- (3.5,1);

        \draw[thick, dashed] (4.5,2) -- (4.5,0) -- (6, 0) -- (6,2);

        \draw[orange, fill=orange] (4.5, 1.5) circle (0.05);

	\node[left] at (4.5, 1.5) {$1$};

        \draw[orange, fill=orange] (6, 1) circle (0.05);

	\node[right] at (6, 1) {$2$};

	\draw[blue, fill=blue] (4.5, 0.5) circle (0.05);

	\node[left] at (4.5, 0.5) {$3$};

        \draw[orange, thick] (6, 1) -- (5.25, 1) -- (5.25, 1.5) -- (4.5, 1.5);

        \node[below] at (5.25,0) {$D_1$};

        \draw[thick, dashed] (7,2) -- (7,0) -- (8.5, 0) -- (8.5,2);

        \draw[orange, fill=orange] (7, 1.5) circle (0.05);

	\node[left] at (7, 1.5) {$1$};

        \draw[orange, fill=orange] (8.5, 1) circle (0.05);

	\node[right] at (8.5, 1) {$2$};

	\draw[blue, fill=blue] (7, 0.5) circle (0.05);

	\node[left] at (7, 0.5) {$3$};

        \draw[orange, thick] (8.5, 1) -- (7.75, 1) -- (7.75, 1.5) -- (7, 1.5);

        \draw[orange, thick] (7.75, 1.5) -- (7.75, 2);

        \node[below] at (7.75,0) {$D_2$};

	\end{tikzpicture}

\end{center}

Pictured below is the complete collection $LR(\chi,\epsilon)$ with $\chi$ and $\epsilon$ as above.

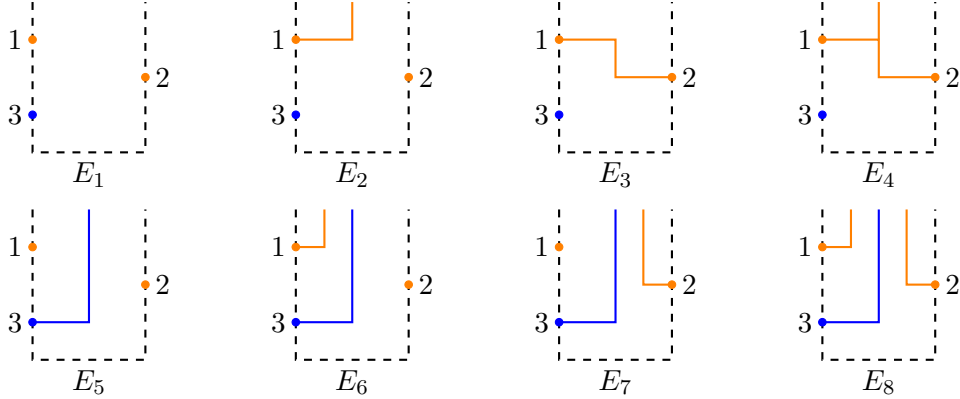
\begin{center}

    \begin{tikzpicture}[baseline]

        \draw[thick, dashed] (0,2) -- (0,0) -- (1.5, 0) -- (1.5,2);

        \draw[orange, fill=orange] (0, 1.5) circle (0.05);

	\node[left] at (0, 1.5) {$1$};

        \draw[orange, fill=orange] (1.5, 1) circle (0.05);

	\node[right] at (1.5, 1) {$2$};

	\draw[blue, fill=blue] (0, 0.5) circle (0.05);

	\node[left] at (0, 0.5) {$3$};

        \node[below] at (0.75,0) {$E_1$};

        \draw[thick, dashed] (3.5,2) -- (3.5,0) -- (5, 0) -- (5,2);

        \draw[orange, fill=orange] (3.5, 1.5) circle (0.05);

	\node[left] at (3.5, 1.5) {$1$};

        \draw[orange, fill=orange] (5, 1) circle (0.05);

	\node[right] at (5, 1) {$2$};

	\draw[blue, fill=blue] (3.5, 0.5) circle (0.05);

	\node[left] at (3.5, 0.5) {$3$};

        \draw[orange, thick] (3.5,1.5) -- (4.25, 1.5) -- (4.25, 2);

        \node[below] at (4.25,0) {$E_2$};

        \draw[thick, dashed] (7,2) -- (7,0) -- (8.5, 0) -- (8.5,2);

        \draw[orange, fill=orange] (7, 1.5) circle (0.05);

	\node[left] at (7, 1.5) {$1$};

        \draw[orange, fill=orange] (8.5, 1) circle (0.05);

	\node[right] at (8.5, 1) {$2$};

	\draw[blue, fill=blue] (7, 0.5) circle (0.05);

	\node[left] at (7, 0.5) {$3$};

        \draw[orange, thick] (8.5, 1) -- (7.75, 1) -- (7.75, 1.5) -- (7, 1.5);

        \node[below] at (7.75,0) {$E_3$};

        \draw[thick, dashed] (10.5,2) -- (10.5,0) -- (12, 0) -- (12,2);

        \draw[orange, fill=orange] (10.5, 1.5) circle (0.05);

	\node[left] at (10.5, 1.5) {$1$};

        \draw[orange, fill=orange] (12, 1) circle (0.05);

	\node[right] at (12, 1) {$2$};

	\draw[blue, fill=blue] (10.5, 0.5) circle (0.05);

	\node[left] at (10.5, 0.5) {$3$};

        \draw[orange, thick] (12, 1) -- (11.25, 1) -- (11.25, 1.5) -- (10.5, 1.5);

        \draw[orange, thick] (11.25, 1.5) -- (11.25, 2);

        \node[below] at (11.25,0) {$E_4$};

	\end{tikzpicture}

\end{center}

\begin{center}

    \begin{tikzpicture}[baseline]

        \draw[thick, dashed] (0,2) -- (0,0) -- (1.5, 0) -- (1.5,2);

        \draw[orange, fill=orange] (0, 1.5) circle (0.05);

	\node[left] at (0, 1.5) {$1$};

        \draw[orange, fill=orange] (1.5, 1) circle (0.05);

	\node[right] at (1.5, 1) {$2$};

	\draw[blue, fill=blue] (0, 0.5) circle (0.05);

	\node[left] at (0, 0.5) {$3$};

        \draw[blue, thick] (0, 0.5) -- (0.75, 0.5) -- (0.75, 2);

        \node[below] at (0.75,0) {$E_5$};

        \draw[thick, dashed] (3.5,2) -- (3.5,0) -- (5, 0) -- (5,2);

        \draw[orange, fill=orange] (3.5, 1.5) circle (0.05);

	\node[left] at (3.5, 1.5) {$1$};

        \draw[orange, fill=orange] (5, 1) circle (0.05);

	\node[right] at (5, 1) {$2$};

	\draw[blue, fill=blue] (3.5, 0.5) circle (0.05);

	\node[left] at (3.5, 0.5) {$3$};

        \draw[orange, thick] (3.5,1.5) -- (3.88, 1.5) -- (3.88, 2);

        \draw[blue, thick] (3.5,0.5) -- (4.25, 0.5) -- (4.25, 2);

        \node[below] at (4.25,0) {$E_6$};

        \draw[thick, dashed] (7,2) -- (7,0) -- (8.5, 0) -- (8.5,2);

        \draw[orange, fill=orange] (7, 1.5) circle (0.05);

	\node[left] at (7, 1.5) {$1$};

        \draw[orange, fill=orange] (8.5, 1) circle (0.05);

	\node[right] at (8.5, 1) {$2$};

	\draw[blue, fill=blue] (7, 0.5) circle (0.05);

	\node[left] at (7, 0.5) {$3$};

        \draw[orange, thick] (8.5, 1) -- (8.12, 1) -- (8.12, 2);

        \draw[blue, thick] (7,0.5) -- (7.75,0.5) -- (7.75, 2);

        \node[below] at (7.75,0) {$E_7$};

        \draw[thick, dashed] (10.5,2) -- (10.5,0) -- (12, 0) -- (12,2);

        \draw[orange, fill=orange] (10.5, 1.5) circle (0.05);

	\node[left] at (10.5, 1.5) {$1$};

        \draw[orange, fill=orange] (12, 1) circle (0.05);

	\node[right] at (12, 1) {$2$};

	\draw[blue, fill=blue] (10.5, 0.5) circle (0.05);

	\node[left] at (10.5, 0.5) {$3$};

        \draw[blue, thick] (10.5, 0.5) -- ( 11.25, 0.5) -- (11.25, 2);

        \draw[orange, thick] (10.5, 1.5) -- (10.88, 1.5) -- (10.88, 2);

        \draw[orange, thick] (12, 1) -- (11.62, 1) -- (11.62, 2);

        \node[below] at (11.25,0) {$E_8$};

	\end{tikzpicture}

\end{center}

\end{example}

\begin{remark}
    Note that any diagram in $LR_0(\chi, \epsilon)$ corresponds to a partition $\pi \in BNC(\chi)$ by defining the blocks in $\pi$ to be those collections of numbers connected by strings in the diagram.

    In the above example $\{E_1, E_3\}$ are the two diagrams in $LR_0(\chi,\epsilon)$ corresponding to the bi-non-crossing partitions $\pi_1 = \{\{1\},\{2\},\{3\}\}$ and $\pi_2 = \{\{1,2\},\{3\}\}$.
\end{remark}

\subsection{An LR-diagram calculus}

Dealing with freeness comes with a certain chaos in performing moment calculations with longer strings of operators. The number of \textit{pieces} that can split off during a process like this can grow very quickly. The $LR$-diagrams provide a useful calculus for keeping track of such pieces.

\begin{definition}
    Let $D_1, D_2 \in LR(\chi,\epsilon)$, we say that $D_1$ is a \textit{lateral refinement} of $D_2$, and write $D_1 \leq_{\text{lat}} D_2$ if $D_2$ can be made from $D_1$ by a sequence of cuts to $D_1$'s spines between ribs (not cutting along the top-gap). 
    
    Let $LR^{\text{lat}}(\chi, \epsilon)$ denote the closure of $LR(\chi,\epsilon)$ under lateral refinement.
\end{definition}

\begin{definition} \label{LRpiecesdef}
    Let $\{(\X_k, \osc{\X}_k, p_k)\}_{k \in K}$ be $B$-$B$-bimodules with specified projections. Let $(\X, \osc{\X}, p)$ be their reduced free product with amalgamation over $B$ and for each $k \in K$ let $\lambda_k$ and $\rho_k$ be respectively the left and right regular representations of $\L(\X_k)$ on $\X$. For each $i \in \{1, \ldots, n\}$, let $T_i \in \L_{\chi(i)}(\X_{\epsilon(i)})$. Define $\mu_i(T_i) = \lambda_{(\epsilon(i))}(T_i)$ if $\chi(i) = \ell$ and $\mu_i(T_i) = \rho_{\epsilon(i)}(T_i)$ if $\chi(i) = r$. For each $D \in LR^{\text{lat}}(\chi, \epsilon)$ we define $E_D(\mu_1(T_1), \ldots, \mu_n(T_n))$ recursively as follows. Apply the same recursive process as in Definition \ref{bimultiplicative} until every block of $D$ has a spine reaching the top. If every block of $D$ has a spine reaching the top gap, enumerate them from left to right according to their spines as $V_1, \ldots, V_m$ with $V_j = \{i_{j,1} < \cdots < i_{j,q_j}\}$, and set 
    \[ E_D(\mu_1(T_1), \ldots, \mu_n(t_n)) = [(1-p_{\epsilon(i_{1,1})})T_{i_{1,1}} \cdots T_{i_{1,q_1}}1_B] \otimes\cdots \otimes [(1-p_{\epsilon(i_{m,1})})T_{i_{m,1}} \cdots T_{i_{m,q_m}}1_B].\]
\end{definition}

\begin{remark} \label{wherearetheLRpieces}
    Given the notations as above, a simple but useful remark is that we have 
    \[ E_D(\mu_1(T_1), \ldots, \mu_n(T_n)) \in \osc{\X}_{\epsilon(i_{1,1})} \otimes \cdots \otimes \osc{\X}_{\epsilon(i_{m,1})}.\]
\end{remark}

We have the following result.

\begin{theorem}[\cite{CNS15}, Lemma 7.1.3] \label{LRcountingtheorem}
    With notation as in Definition \ref{LRpiecesdef}, 
    \[ \mu_1(T_1) \cdots \mu_n(T_n)1_B = \sum_{D \in LR^{\text{lat}}(\chi, \epsilon)} c_D E_D(\mu_1(T_1),\ldots, \mu_n(T_n)),\]
    where $c_D$ are constants depending only on the diagram $D$.
\end{theorem}

In Lemma 7.1.3 of \cite{CNS15}, the constants $c_D$ are made explicit, but we have no need for such precision. Keeping them as abstract constants will suffice and help to keep the notational noise to a minimum.

In fact, the proof of the above result can be extended more generally, as it details a recursive way to add on to the diagram-indexed vectors, $E_D$,  with more operators. We start with some notation to help us keep track.

\begin{definition}
    Let $i \leq n \in \bN$. Let $\tilde{\chi} = \chi|_{\{i, \ldots, n\}}$, $\tilde{\epsilon} = \epsilon|_{\{i, \ldots, n\}}$, and $D_0 \in LR^{\text{lat}}(\tilde{\chi}, \tilde{\epsilon})$. Say a diagram $D \in LR^{\text{lat}}(\chi, \epsilon)$ is a \textit{$\chi$-extension} of $D_0$ if there is some diagram $D' \in LR^{\text{lat}}(\chi,\epsilon)$ with $D' |_{\{i,\ldots, n\}} = D_0$ and $D$ can be laterally refined from $D'$ by only making spine-cuts (between ribs) above $i$ in the diagram $D'$.

    If $S \subseteq LR^{\text{lat}}(\tilde{\chi},\tilde{\epsilon})$, denote by $S^\chi$, the set of diagrams in $LR^{\text{lat}}(\chi,\epsilon)$ which are $\chi$-extensions of diagrams in $S$.
\end{definition}

\begin{remark} \label{extofext}
    It is immediate from the definition above that if $i_1 \leq i_2$, $\chi_0 = \chi |_{\{i_1, \ldots, n\}}$ and $\tilde{\chi} = \chi_0 |_{\{i_2, \ldots, n\}}$, then 
    \[ (S^{\chi_0})^{\chi} = S^\chi\]
    for any $S \subseteq LR^{\text{lat}}(\tilde{\chi}, \tilde{\epsilon})$
\end{remark}

The purpose of this definition is that it contains the process described in the proof of Theorem \ref{LRcountingtheorem} for applying new (left and right) operators to diagram-indexed vectors. With this notation in hand, we extend Theorem \ref{LRcountingtheorem} to the following result.

\begin{lemma} \label{extendedLRcountingtheorem}
    Let $i \leq n$ and other notations be as above. Then for any $S \subseteq LR^{\text{lat}}(\tilde{\chi}, \tilde{\epsilon})$ and constants $c_{D,S}$ for each $D \in S$, there are constants $c_{D,S}'$ such that
    \[ \mu_1(T_1) \cdots \mu_{k-1}(T_{i-1}) \left( \sum_{D \in S} c_{D,S} E_D(\mu_k(T_i), \ldots, \mu_n(T_n))\right) = \sum_{D \in S^{\chi}}c_{D,S}'E_D(\mu_1(T_1), \ldots, \mu_n(T_n))\]
\end{lemma}

\begin{proof}
    The proof is the same as for \cite[Lemma 7.1.3]{CNS15}, as the process inductively added on to existing diagrams. The only part that differs now is that we may end up with different constants scaling the vectors $E_D$, since we are starting with a potentially smaller collection of diagrams than in Lemma 7.1.3. of \cite{CNS15}. In the same way however, these constants depend only on the diagram $D$ and the collection of diagrams $S$ that we use.
\end{proof}

\begin{notation} \label{boolprojannihilatesLR}
    We want to combine this result with our Boolean projections. To this end, for each $k \in K$, we define the set 
    \[ LR^{\text{lat}}(\chi, \epsilon, k) := \{ D \in LR^{\text{lat}}(\chi, \epsilon) \mid \text{if a spine of $D$ reaches the top gap then it has colour $k$} \}.\]
    By Remark \ref{wherearetheLRpieces}, this is precisely the collection of $LR$-diagrams which do not get annihilated by the Boolean projection $P_k$.

    It is worth noting that since lateral refinements may only cut spines between ribs (not at the top gap), it follows that $D \in LR^{\text{lat}}(\chi, \epsilon, k)$ if $D$ has either no spines reaching the top gap, or exactly one spine reaching the top gap of colour $\epsilon(k)$.

    We denote the complement of $LR^{\text{lat}}(\chi,\epsilon,k)$ in $LR^{\text{lat}}(\chi,\epsilon)$ simply by $LR^{\text{lat}}(\chi,\epsilon,k)^c$. This is the collection of diagrams with at least one spine reaching the top gap with a different colour than $k$.
\end{notation}

Combining Theorem \ref{LRcountingtheorem} with Remark \ref{wherearetheLRpieces} yields the following.

\begin{lemma} \label{boolprojLR}
    With notations as in Definition \ref{LRpiecesdef}, let $k \in K$ and let $P_k$ be a Boolean projection as in Definition \ref{boolprojdef}.
    Then we have 
    \[ P_k \sum_{D \in S} c_{D,S} E_D(\mu_1(T_1), \ldots, \mu_n(T_n)) = \sum_{D \in LR^{\text{lat}}(\chi, \epsilon,k) \cap S} c_{D,S} E_D(\mu_1(T_1),\ldots, \mu_n(T_n)),\]
    and consequently
    \[ (1-P_k)\sum_{D \in S} c_{D,S} E_D(\mu_1(T_1), \ldots, \mu_n(T_n)) = \sum_{D \in LR^{\text{lat}}(\chi, \epsilon,k)^c \cap S} c_{D,S} E_D(\mu_1(T_1),\ldots, \mu_n(T_n)).\]
\end{lemma}

We can combine our observations to get the following technical lemma which serves as an extension of the LR-diagram calculus, but now including the Boolean projections.

\begin{lemma} \label{keyprojLRlemma}
    With notations as above, for $m \geq 0$, let $1 \leq i_1 < i_2 < \cdots < i_m \leq n$ and define 
    \[ \mu_j'(T_j) := \begin{cases} \mu_j(T_j) & \text{ if } j \notin \{i_1, \ldots, i_m\}\\ P_{\epsilon(j)}\mu_j(T_j) & \text{ otherwise}\end{cases} .\]
    Then there is some $S \subseteq LR^{\text{lat}}(\chi,\epsilon)$ and constants $c_{D,S}$ such that
    \[ \mu_1(T_1) \cdots \mu_n(T_n)1_B = \mu_1'(T_1) \cdots \mu_n'(T_n)1_B \\
    + \sum_{D \in S} c_{D,S} E_D(\mu_1(T_1),\ldots, \mu_n(T_n)), \]
    where $S \subseteq \bigcup_{j = 1}^m S_{i_j}^\chi$, and 
    $S_{i_j} = LR^{\text{lat}}(\chi|_{\{i_j, \ldots, n\}},\epsilon|_{\{i_j, \ldots, n\}}, i_j)^c$.
\end{lemma}

\begin{proof}
    We proceed by induction on $m$, letting $n \in \bN$ be otherwise arbitrary. When $m = 0$, we notice that $\mu_j'(T_j) = \mu_j(T_j)$ for all $1 \leq j \leq n$ and so we can take $S = \emptyset$.

    Now suppose for $m \geq 0$ the claim holds, fix $n \in \bN$ with $n \geq m+1$ and let $1 \leq i_1 < \cdots < i_{m+1} \leq n$. For convenience, let $\eta_a = \mu_1(T_1) \cdots P_{\epsilon(i_1)}\mu_{i_1}(T_{i_1}) \cdots\mu_n(T_n)1_B$ and $\eta_b = \mu_1(T_1) \cdots (1-P_{\epsilon(i_1)})\mu_{i_1}(T_{i_1}) \cdots\mu_n(T_n)1_B $. Then we have
    \[
    \mu_1(T_1) \cdots \mu_n(T_n)1_B = \eta_a + \eta_b 
    \]

    We take each piece individually. Let $\tilde{\chi} = \chi|_{\{i_1, \ldots, n\}}$ and $\tilde{\epsilon} = \epsilon|_{\{i_1, \ldots, n\}}$. By hypothesis, there is some $S_0 \subseteq \bigcup_{j = 2}^{m+1} S_{i_j}^{\tilde{\chi}}$ where
    \begin{align*} \mu_{i_1}(T_{i_1}) \cdots\mu_n(T_n)1_B = &\mu_{i_1}(T_{i_1})\mu_{i_1+1}'(T_{i_1+1}) \cdots \mu_n'(T_n)1_B 
    \\ 
    &\ \ + \sum_{D \in S_0}c_{D,S_0} E_D(\mu_{i_1}(T_{i_1}), \ldots, \mu_n(T_n)).
    \end{align*}

    Now by definition, $\mu_1(T_1) \cdots P_{\epsilon(i_1)}\mu_{i_1}(T_{i_1}) = \mu_1'(T_1) \cdots \mu_{i_1}'(T_{i_1})$, so that
    \[ \eta_a = \mu_1(T_1) \cdots P_{\epsilon(i_1)}\mu_{i_1}(T_{i_1})\cdots \mu_n(T_n) 1_B = \mu_1'(T_1) \cdots \mu_n'(T_n)1_B + \eta_a'\]
    where 
    \[ \eta_a' = \mu_1(T_1) \cdots P_{\epsilon(i_1)} \sum_{D \in S_0} c_{D,S_0}E_D(\mu_{i_1}(T_{i_1}), \ldots, \mu_n(T_n)).\]
    Applying Lemmas \ref{extendedLRcountingtheorem} and \ref{boolprojLR},
    \[ \eta_a' = \sum_{D \in (LR^{\text{lat}}(\tilde{\chi},\tilde{\epsilon},i_1)\cap S_0)^\chi} c_{D,S_0} E_D(\mu_1(T_1), \ldots, \mu_n(T_n)). \]
    It is readily verified, making use of Remark \ref{extofext} that 
    \[ S_a := (LR^{\text{lat}}(\tilde{\chi}, \tilde{\epsilon}, i_1) \cap S_0)^\chi \subseteq \bigcup_{j = 1}^{k+1} S_{i_j}.\]

    As for $\eta_b$, note that Theorem \ref{LRcountingtheorem} says
    \[ \mu_{i_1}(T_{i_1}) \cdots\mu_n(T_n)1_B = \sum_{D \in LR^{\text{lat}}(\chi,\epsilon)} c_DE_D(\mu_{i_1}(T_{i_1}), \ldots,\mu_n(T_n)),\]
    so by Lemma \ref{boolprojLR},
    \[ (1-P_{\epsilon(i_1)})\mu_{i_1}(T_{i_1}) \cdots\mu_n(T_n)1_B = \sum_{D \in LR^{\text{lat}}(\tilde{\chi},\tilde{\epsilon}, i_1)^c} c_DE_D(\mu_{i_1}(T_{i_1}), \ldots, \mu_n(T_n)).\]
    Letting $S_b = (LR^{\text{lat}}(\tilde{\chi}, \tilde{\epsilon}, i_1)^c)^{\chi}$ and appending our initial operators on grants us 
    \begin{align*}
        \eta_b = &\mu_1(T_1)\cdots (1-P_{\epsilon(i_1)})\mu_{i_1}(T_{i_1}) \cdots\mu_n(T_n)1_B \\
        = &\sum_{D \in S_b} c_{D,S_b}E_D(\mu_1(T_1), \ldots, \mu_n(T_n)).
    \end{align*}
    Notice immediately from its definition that $S_b \subseteq \bigcup_{j = 1}^{k+1}S_{i_j}^\chi$. So if we let $S = S_a \cup S_b$ we have $S \subseteq \bigcup_{j = 1}^{k+1} S_{i_j}^\chi$. Define, for each $D \in S$, $c_{D,S} := c_{D, S_a} + c_{D, S_b}$ where $c_{D, S_a} = 0$ if $D \notin S_a$ and $c_{D, S_b} = 0$ if $D \notin S_b$, then
    \[ \eta_a' + \eta_b = \sum_{D \in S} c_{D,S}E_D(\mu_1(T_1), \ldots, \mu_n(T_n))\]
    so that
    \begin{align*}
        \mu_1(T_1) \cdots \mu_n(T_n)1_B &= \mu_1'(T_1) \cdots \mu_n'(T_n)1_B + \eta_a' + \eta_b \\
        &= \mu_1'(T_1)\cdots \mu_n'(T_n)1_B + \sum_{D \in S}c_{D,S}E_D(\mu_1(T_1), \ldots, \mu_n(T_n)),
    \end{align*}
    as desired.
\end{proof}

\section{Free-Free-Boolean Independence with Amalgamation}

In \cite{liu17}, Liu introduced a notion of \textit{free-free-Boolean} independence for triples of algebras. In this section we generalize this definition to the amalgamated setting, and then in the following section we will show how this independence can arise in the context of bi-free probability, in both a representational sense, and combinatorially with the bi-free cumulants. 

\begin{definition}
    A triple of $B$-faces in a $B$-$B$-noncommutative probability space $(\A, E_\A, \varepsilon)$ is a triple $(A,C,D)$ where $A$, $C$, and $D$ are subalgebras of $\A$ with $A$ and $C$ unital such that 
    \[ \varepsilon( B \otimes 1_B) \subseteq A \subseteq \A_\ell, \quad \varepsilon(1_B \otimes B) \subseteq C \subseteq \A_r, \quad \text{ and } \varepsilon(B \otimes 1_B) D \varepsilon(B \otimes 1_B) \subseteq D \subseteq \A_\ell. \]
\end{definition}

\begin{definition}
    Let $\Gamma = \{(A_k,C_k,D_k)\}_{k \in K}$ be a family of triples of $B$-faces in $(\A, E_\A, \varepsilon)$. The \textit{joint distribution} of $\Gamma$ is the $B$-valued functional $\mu_\Gamma = E_\A|_{\text{alg}(\Gamma)}$.
\end{definition}

\begin{remark}
    If we know how the distribution behaves on words of operators from $\Gamma$, then we know the total distribution by linearity. To this end, we will exclusively discuss the expectations of words, called \textit{moments}. An arbitrary word of operators from $\Gamma$ will require an arbitrary $n \in \bN$, a function $\chi : \{1, \ldots, n\} \to \{\ell, r, b\}$ and a colouring $\epsilon : \{1, \ldots, n \} \to K$, as well as elements $Z_1, \ldots, Z_n$ such that 
    \[ Z_i \in \begin{cases} A_{\epsilon(i)} \text{ if } \chi(i) = \ell , \\ C_{\epsilon(i)} \text{ if } \chi(i) = r, \\ D_{\epsilon(i)} \text{ if }\chi(i) = b\end{cases}.\]
    Then $\mu_\Gamma$ is completely determined by the moments
    \[ \mu_\Gamma(Z_1 \cdots Z_n) = E_\A(Z_1 \cdots Z_n).\] 
\end{remark}

\begin{definition} \label{ffbdef}
    Let $(\A, E_\A, \varepsilon)$ be a $B$-$B$-noncommutative probability space. Given a family  $\Gamma = \{(A_k,C_k, D_k)\}_{k \in K}$ of triples of $B$-faces in $(\A, E_\A, \varepsilon)$, suppose there exists a family of $B$-$B$-bimodules with specified $B$-projections $\{(\X_k,\osc{\X}_k, p_k)\}_{k \in K}$ and unital homomorphisms $\ell_k : A_k \to \L_\ell(\X_k)$ and $r_k : C_k \to \L_r(\X_k)$, and (not necessarily unital) homomorphisms $m_k : D_k \to \L(\X_k)$. Let $(\X, \osc{\X}, p)$ be the reduced free product of the family $\{(\X_k, \osc{\X}_k, p_k)\}_{k \in K}$. Let $\lambda_k$ and $\rho_k$ denote the corresponding left and right regular representations of $\L(\X_k)$ on $\X$, and $P_k$ denote the $k^{\text{th}}$ Boolean projection. The family $\Gamma$ is said to be \textit{free-free-Boolean independent with amalgamation over $B$} if the joint distribution $\mu_\Gamma$ is equal to the joint distribution of the family 
    $\{(\lambda_k(\ell_k(A_k)), \rho_k(r_k(C_k)), P_k(\lambda_k(m_k(D_k)))P_k)\}_{k \in K}$ in $(\L(\X), E_{\L(\X)})$.
\end{definition}

\begin{notation}
    The triples of $B$-faces throughout the remainder of this work will be related to free-free-Boolean families. 
    Thus, we will simplify our notation by denoting a triple of $B$-faces of $\A$ as $(A^\ell, A^r, A^b)$.
\end{notation}

\begin{remark} \label{ffbimpliesbifree}
    Notice that from the definition above, if $\Gamma = \{(A^\ell_k, A^r_k,A^b_k)\}_{k \in K}$ is a free-free-Boolean independent family, then $\pi = \{(A^\ell_k,A^r_k)\}_{k \in K}$ is a bi-free family. It can be easily verified also that the family $\{ A^b_k\}_{k \in K}$ is a Boolean independent family.

    Furthermore, it is readily verified that in fact, the family of pairs of algebras $\{(A_k^\ell, A_k^b)\}_{k \in K}$ are free-Boolean independent with amalgamation over $B$ in the sense of \cite{liu+zhong}, noting that $A_k^b$ is a $B$-$B$-bimodule by left and right multiplication by $\epsilon(B \otimes 1_B)$.

    Lastly, in the case where $B=\bC$, this definition readily conforms with a free-free-Boolean independent family as defined in \cite{liu17}. Indeed, a triple of faces so-defined can be found to embed faithfully in a single non-commutative probability space by constructing a suitable free product.
\end{remark}

%\begin{remark}
%    There is a slight concession being made in the definition of a triple of $B$-faces and eventually the definition of free-free-Boolean independence, in that we are demanding that the algebras $\A_k^b$ consist of only those operators from $\A_\ell$ and is stable under left and right multiplication by $\epsilon(B \otimes 1_B)$.

%    This assumption does not appear to be too restrictive. One could, for example, formulate equivalent definitions using right operators instead of left. In addition, in the case where $B = \bC$, there is no impact from this assumption at all.
%\end{remark}

\begin{remark}
    A natural property of independence of a family of algebras or tuples of algebras should be that the overall joint distribution should depend only on the individual (joint) distributions within the family. This is true of bi-freeness, but it is also readily seen to  be true of free-free-Boolean independence. Indeed, one can observe that by rearranging the equation in Lemma \ref{keyprojLRlemma} and using bi-multiplicativity, one sees that the expectation of a word coming from a free-free-Boolean independent family is determined from moments on the individual triples. 
\end{remark}

As a simple consequence of this, one can choose any space on which to represent our free-free-boolean independent family without fear of losing our independence.

\begin{proposition}
    Let $\Gamma = \{ (A^\ell_k, A^r_k, A^b_k)\}_{k \in K}$ be a family of triples of $B$-faces in a $B$-$B$-non-commutative probability space $(\A,E_\A, \epsilon)$. If we have for each $k \in K$, a $B$-$B$-bimodule $(\X'_k, \osc{X}'_k, p'_k)$ and homomorphisms $\ell_k' : A^\ell_k \to \L_\ell(\X_k')$, $r_k' : A^r_k \to \L_r(\X_k')$ and $m_k' : A^b_k \to \L(\X_k')$, such that that the joint distribution of the triple $(A^\ell_k, A^r_k, A^b_k)$ equals the joint distribution of the image $(\ell_k'(A^\ell_k), r_k'(A^r_k), m_k'(A^b_k))$ for each $k \in K$, then $\Gamma$ is free-free-Boolean independent if and only if $\Gamma' = \{(\ell_k'(A^\ell_k), r_k'(A^r_k), m_k'(A^b_k))\}_{k \in K}$ is free-free-Boolean independent. In this case, the joint distribution of $\Gamma$ is equal to the joint distribution of $\Gamma'$.
\end{proposition}

In particular, we will make use of the $B$-$B$-bimodule with specified $B$-projection from \ref{bbbimoduleconstruction}.

\section{How Free-Free-Boolean Independence with Amalgamation Arises in Bi-Free Probability}

Now we turn our attention to how free-free-Boolean independence embeds in bi-free probability. 

\subsection{Abstract Bi-Free Free-Free-Boolean Structures}

We begin by defining a structure within with bi-free independence that produces free-free-Boolean independence. We proceed with inspiration from \cite{Sko14}. Since we are trying to model a free-free-Boolean independent system, we upgrade Skoufranis' bi-free Boolean $B$-systems with additional bi-free pairs of $B$-faces. 

\begin{definition} \label{ffbsystem}
    Let $(\A, E_\A, \varepsilon)$ be a $B$-$B$-non-commutative probability space and fix a family $\{ (A_k^\ell, A_k^r, C_k, D_k)\}_{k \in K}$ of quadruples of subalgebras of $\A$ such that each of $(A_k^\ell, A_k^r)$ are pairs of $B$-faces of $\A$ and the family of pairs $\{ (\text{alg}(A_k^\ell, C_k), \text{alg}(A_k^r, D_k))\}_{k \in K}$ are bi-free over $B$. 
    For each $k \in K$, let $C_k' \subseteq C_k$ and $D_k' \subseteq D_k \cap \A_\ell$ be subsets such that $L_bC_k' \subseteq C_k'$ and $C_k' L_b \subseteq C_k'$ for all $b \in B$. We say that $\{(A_k^\ell, A_k^r, C_k', D_k')\}_{k \in K}$ is a \textit{bi-free ffb $B$-system} if for all $n \geq 0$ and $k \in K$,
    \begin{enumerate}
        \item $C_k' \bA_k C_k' = \{0\} = D_k' \bA_k D_k' $, 
        \item $E_\A( \bA_k C_k' \bA_k (D_k' \bA_k C_k')^n \bA_k) = 0$, and
        \item $E_\A( \bA_k D_k' \bA_k(C_k' \bA_k D_k')^n \bA_k) = 0$
        
    \end{enumerate}
    where $\bA_k := \text{alg}(A_k^\ell, A_k^r)$.
    
\end{definition}

\begin{remark}
    This extends the notion of a bi-free Boolean $B$-system from \cite{Sko14}. In particular, since the algebras $\A_k^\ell$ and $\A_k^r$ are unital, it follows that the family of pairs $\{(C_k', D_k')\}_{k \in K}$ is a bi-free boolean $B$-system. It was shown in \cite{Sko14} that therefore the algebras $\{\text{alg}(C_k'D_k')\}_{k \in K}$ Boolean independent with amalgamation over $B$.
\end{remark}

Such a bi-free ffb $B$-system as defined in Definition \ref{ffbsystem} forms the triple of $B$-faces
\[ \{ (A_k^\ell, A_k^r, \text{alg}(C_k'D_k'))\}_{k \in K}.\]
Indeed, since $D_k' \in \A_r \cap \A_\ell$, and $L_b C_k', C_k'L_b \subseteq C_k'$, we have that 
\[ \varepsilon(B \otimes 1_B) \text{alg}(C_k'D_k')\varepsilon(B\otimes 1_B) \subseteq \text{alg}(C_k'D_k').\]
We will show that such a triple of faces is free-free-Boolean independent with amalgamation over $B$. But first, let's see how such a collection can house any free-free-Boolean independent family.

\subsection{The Main Construction}

The following construction motivates our definition of bi-free ffb $B$-system. This will allow us to embed free-free-Boolean families into bi-free families. Our strategy follows the idea of Construction 4.2.3 of \cite{Sko14}, but tweaked to use an appropriate $B$-$B$-bimodule with specified $B$-projection.

Let us fix a $B$-$B$-noncommutative probability space $(\A, E_\A, \varepsilon)$. Fix a free-free-Boolean independent family $\{(A_k^\ell, A_k^r, A_k^b)\}_{k \in K}$ in $\A$. Let 
\[ \A' = B \oplus \ker(E_\A)/ \text{span}\{TL_b - TR_b : T \in \A, b \in B\}\]
as in Theorem \ref{faithfulbimodule}, and let $\theta : \A \to \A'$ be the corresponding embedding. Then let $(\Y_k, \osc{\Y_k},q_k)$ be a copy of the $B$-$B$-bimodule with specified $B$-projection as in Remark and Notation \ref{bbbimoduleconstruction}. That is, $\Y_k = \A' \oplus \A'$ and 
\[ \osc{\Y}_k = \ker(E_\A)/ \text{span}\{TL_b - TR_b : T \in \A, b \in B\} \oplus \A' .\]

Now, for each $k \in K$ and any $Z \in A_k^b$, consider the operators $T_{Z}, S_{1_B} \in \L(\Y_k)$ defined by
\[ T_Z (\xi_1 \oplus \xi_2) = \theta(Z) \xi_2 \oplus 0 \quad \text{ and } \quad S_{1_B} (\xi_1 \oplus \xi_2) = 0 \oplus \xi_1, \]
for all $\xi_1, \xi_2 \in \A'$. It is simple to see that $S_{1_B} \in \L_\ell(\Y_k) \cap \L_r(\Y_k)$ and that $T_Z \in \L_\ell(\Y_k)$. Indeed, let $\xi_1, \xi_2 \in \A'$ and $b \in B$, then 
\[ S_{1_B}((\xi_1 \oplus \xi_2) \cdot b) = S_{1_B}(\theta(R_b)\xi_1 \oplus \theta(R_b) \xi_2) = 0 \oplus \theta(R_b)\xi_1 = (0 \oplus \xi_1) \cdot b = (S_{1_B}(\xi_1 \oplus \xi_2)) \cdot b, \]
showing that $S_{1_B} \in \L_\ell(\Y_k)$. One can show similarly that $S_{1_B} \in \L_r(\Y_k)$. Showing that $T_Z \in \L_\ell(\Y_k)$ is similar, but makes use of the fact that $Z \in A_k^b \subseteq \A_\ell$ so that $Z R_b = R_b Z$.

Also, for any $Z \in \A$ and any $k \in K$, define $D_Z \in \L(\Y_k)$ by
\[ D_Z(\xi_1 \oplus \xi_2) = \theta(Z)\xi_1 \oplus \theta(Z) \xi_2.\]
If $Z \in \A_\ell$, then $D_Z \in \L_\ell(\Y_k)$. Similarly if $Z \in \A_r$ then $D_Z \in \L_r(\Y_k)$.

Let $(\Y, \osc{\Y}, q)$ be  the reduced free product of the family $\{(\Y_k, \osc{\Y}_k, q_k)\}_{k \in K}$. For each $k \in K$ let $\lambda_k$ and $\rho_k$ be respectively the left and right regular representations of $\L(\Y_k)$ on $\Y$. For $Z \in \A$, define
\[ T_{k,Z} := \lambda_k(T_Z) \quad \text{ and } S_{k, 1_B} := \rho_k(S_{1_B}).\]

For each $k \in K$, define $C'_k = \{ T_{k,Z} \ |\  Z \in A_k^b\}$ and $D'_k = \{ S_{k,1_B}\}$. Let 
\[ \tilde{A}_k^\ell = \{ \lambda_k(D_Z) \mid Z \in A_k^\ell\} \quad  \text{ and } \quad \tilde{A}_k^r = \{ \rho_k(D_Z) \mid Z \in A_k^r\}.\]
Notice that for each $k \in K$, $(\tilde{A}_k^\ell, \tilde{A}_k^r)$ is a pair of $B$-faces of $\L(\Y)$.

Then we check that the family $\{ (\tilde{A}_k^\ell, \tilde{A}_k^r,C_k', D_k')\}_{k \in K} $ is a bi-free ffb $B$-system in $\L(\Y)$. For property (1), fix $k \in K$ and $n \in \bN$, $\chi : \{1, \ldots, n\} \to \{\ell, r \}$ and let $Z_1, \ldots, Z_n \in \A$ such that $Z_i \in A_k^{\chi(i)}$. For convenience, let $Z = Z_1 \cdots Z_n$ so that $D_Z = D_{Z_1} \cdots D_{Z_n}$. Then for any $\xi_1 \oplus \xi_2 \in \Y_k$, 
\[ S_{1_B} D_Z S_{1_B} (\xi_1 \oplus \xi_2) = S_{1_B}D_Z( 0 \oplus \xi_1) = S_{1_B}( 0 \oplus \theta(Z) \xi_1) = 0. \]
Similarly, for any $A_1, A_2 \in A_k^b$, 
\[ T_{A_1}D_ZT_{A_2} (\xi_1 \oplus \xi_2) = T_{A_1} D_Z (\theta(A_2)\xi_2 \oplus 0) = T_{A_1} (\theta(Z)\theta(A_2) \xi_2 \oplus 0) = 0. \]

Now stepping up to the free product, for each $1 \leq i \leq n$ let $D_{Z_i}'$ denote $\lambda_k(D_{Z_i})$ if $\chi(i) = \ell$ and $\rho_k(D_{Z_i})$ if $\chi(i) = r$. Since $Z_i \in A_k^{\chi(i)}$ for each $1 \leq i \leq n$, if $\eta \in \Y_k \subseteq \Y$ then
\[ \rho_k(S_{1_B})D_{Z_1}'\cdots D_{Z_n}'\rho_k(S_{1_B})\eta = S_{1_B}D_{Z_1} \cdots D_{Z_n}S_{1_B}\eta = 0,\]
and 
\[ \lambda_k(T_{A_1})D_{Z_1}'\cdots D_{Z_n}'\lambda_k(T_{A_2})\eta = T_{A_1}D_{Z_1}\cdots D_{Z_n} T_{A_2}\eta = 0. \]
If $\eta \in \Y \ominus\Y_k$, then we have that for any $X_1, X_2 \in \L(\Y_k)$
\[ \lambda_k(X_1)\rho_k(X_2)\eta = \rho_k(X_2)\lambda_k(X_1)\eta.\]
With this in mind, form the (ordered) products
\[ Z^\ell = \prod_{\substack{1 \leq i \leq n \\ \chi(i) = \ell}} Z_i \quad \text{ and } \quad Z^r = \prod_{\substack{1 \leq i \leq n \\ \chi(i) = r}} Z_i.\]
Then, by commuting lefts with rights in their action on $\eta \in \Y \ominus \Y_k$ and combining with our understanding of the annihilating action from above, 
\[ \rho_k(S_{1_B})D_{Z_1}'\cdots D_{Z_n}'\rho_k(S_{1_B}) \eta = \lambda_k(D_{Z^\ell})(\rho_k(S_{1_B}D_{Z^r}S_{1_B})\eta )= \lambda_k(D_{Z^\ell})0 = 0. \]
Similarly, 
\[ \lambda_k(T_{A_1}) D_{Z_1}' \cdots D_{Z_n}'\lambda_k(T_{A_2})\eta = \rho_k(D_{Z^r})(\lambda_k(T_{A_1}D_{Z^\ell}T_{A_2}) \eta) = \rho_k(D_{Z^r}) 0 = 0.\]
Combining these observations together, we get that property (1) holds. 

Property (2) of Definition \ref{ffbsystem} is easy to see, as if $Z$ is as above and $A \in A_k^b$ then
\[ T_AD_Z(1_B \oplus 0) = T_A(\theta(Z)1_B \oplus 0) = 0.\]
Passing to the free product by left and right regular representations does not make a difference here as we are only acting on the element $1_B \in \Y_k$.

Lastly, let $n \in \bN$, $A_1, \ldots, A_n \in A_k^b$, and $Z_1, Z_2, \ldots, Z_{n+1} \in \text{alg}( A_k^\ell, A_k^r)$. Note that 
\begin{align*}
T_{A_n}D_{Z_n}S_{1_B}D_{Z_{n+1}}1_B &= T_{A_n}D_{Z_n}S_{1_B}(\theta(Z_{n+1})1_B \oplus 0) \\
  &= T_{A_n}D_{Z_n}(0 \oplus \theta(Z_{n+1})1_B) \\
  &= T_{A_n}(0 \oplus \theta(Z_n)\theta(Z_{n+1})1_B)\\
  &= \theta(A_n)\theta(Z_n)\theta(Z_{n+1})1_B \oplus 0\\
  &= \theta(A_nZ_nZ_{n+1})1_B.
\end{align*}

So inductively,
\begin{align*}
D_{Z_1}S_{1_B} D_{Z_2} T_{A_1} D_{Z_3} S_{1_B} \cdots T_{A_n}D_{Z_{n+1}}S_{1_B}D_{Z_{n+1}}1_B &= D_{Z_1}S_{1_B} (\theta(Z_2A_1Z_3 \cdots A_nZ_{n+2})1_B \oplus 0)\\
  &= D_{Z_1}(0 \oplus \theta(Z_2A_1Z_3 \cdots A_nZ_{n+2})1_B) \\
  &= 0 \oplus \theta(Z_1Z_2A_1 \cdots A_nZ_{n+2})1_B,
\end{align*}
which has expectation 0. So property (3) follows as well.

So our family $\{(\tilde{A}_k^\ell, \tilde{A}_k^r, C_k', D_k')\}_{k \in K}$ is a bi-free ffb $B$ system. We get the following.

\begin{proposition} \label{partialffb}
    Given a free-free-Boolean independent family $\Gamma = \{ (A_k^\ell, A_k^r, A_k^b)\}_{k \in K}$ in a $B$-$B$-non-commutative probability space $(\A, E_\A, \varepsilon)$, there exists a bi-free ffb $B$-system $\{ (\tilde{A}_k^\ell, \tilde{A}_k^r, C_k', D_k')\}_{k \in K}$ inside a $B$-$B$-non-commutative probability space $(\A', E_{\A'}, \varepsilon')$ and for each $k \in K$ there are injective $B$-linear maps $\alpha_{\ell,k} : A_k^\ell \to \tilde{A}_k^\ell$, $\alpha_{r,k} : A_k^r \to \tilde{A}_k^r$ and $\alpha_{b,k} : A_k^b \to \text{alg}(C_k'D_k')$ such that whenever $n \in \bN$, $\chi :\{ 1 , \ldots, n\} \to \{ \ell, r, b\}$, and $Z_i \in A_{k}^\chi(i)$ for all $1 \leq i \leq n$, we have 
    \[ E_{\A'}(\alpha_{\chi(1),k}(Z_1) \cdots \alpha_{\chi(n),k}(Z_n)) = E_\A(Z_1 \cdots Z_n).\]
\end{proposition}

\begin{proof}
    Using the notation as in the construction above, we let $\A' = \L(\Y)$, $E_{\A'} = E_{\L(\Y)}$, and $\varepsilon'(b_1 \otimes b_2) = L_{b_1}R_{b_2} \in \L(\Y)$. Then define for each $k \in K$ 
    \[ \alpha_{\ell, k}(Z) = \lambda_k(D_Z), \quad \alpha_{r,k}(Z) = \rho_k(D_Z), \quad \text{ and } \quad \alpha_{b,k}(Z) = T_{k,Z}S_{k,1_B}\]
    for each $Z \in A_k^\ell, A_k^r, A_k^b$ respectively. Then each of these maps are $B$-linear and the family $\{ ( \alpha_{\ell, k}(A_k^\ell),\alpha_{r,k}(A_k^r), \{T_{k,Z}:Z \in A_k^b\}, \{S_{k,1_B}\}\}_{k \in K}$ is a bi-free ffb $B$-system, as we showed above. That these maps are jointly expectation preserving can be seen to also come from our calculation for property (3) above.
\end{proof}

This result lacks a punch as we have not yet shown that the free-free-Boolean independence is preserved, not allowing us to extend to the joint distribution of $\Gamma$. In order to demonstrate independence, we make use of the bi-free cumulants. We put the finishing touches on this result below with Theorem \ref{abstractstructuresareffb}.

\subsection{Bi-Free Cumulants for Free-Free-Boolean Independence}

Next, we make use of the bi-free cumulants to form free-free-Boolean cumulants. To do so, with inspiration coming from our construction above, we will split the Boolean-acting operators into a left and a right piece. So we need a way of connecting any colouring $\hat{\chi} : \{1, \ldots, n\} \to \{\ell, r, b\}$ with an appropriate $\ell$-$r$ colouring. 

\begin{definition}
    Given a map $\hat{\chi} : \{1,\ldots, n\} \to \{\ell, r, b\}$, define the \textit{$\ell$-$r$-replacement} of $\hat{\chi}$, to be the map $\upchi:\{1, \ldots, n + |\hat{\chi}^{-1}(\{b\})|\} \to \{\ell, r\}$ defined as follows. Let $f_{\hat{\chi}}(i) = i + |\hat{\chi}^{-1}(\{b\}) \cap [1, i-1]|$. Then, 
    \[ \upchi(i) := \begin{cases} \hat{\chi}(f_{\hat{\chi}}^{-1}(i)) & \text{ if } i \in f_{\hat{\chi}}(\{1, \ldots, n\} \setminus \hat{\chi}^{-1}(\{b\})) \\ 
    \ell &\text{ if } i \in f_{\hat{\chi}}(\hat{\chi}^{-1}(\{b\}) \\
    r &\text{ otherwise}\end{cases}. \]
\end{definition}

So in other words, $\upchi$ is $\ell$-$r$-replaced, if it replaces all instances of $b$ with $\ell, r$ in order. To do so, we must inflate the domain of $\hat{\chi}$ appropriately.

\begin{example}
    As a simple example, consider $\hat{\chi} = \ell b r b$ then the $\ell$-$r$-replacement of $\hat{\chi}$ is the map 
    \[ \chi = \ell ( \ell r) r (\ell r) = \ell \ell r r \ell r. \]
\end{example}

Now we collect all of those partitions which respect this replacement. That is, we want to keep those elements together in our partitions who came from a `$b$' colouring. We will show shortly that partition-moments which separate these elements will vanish.

\begin{definition}
    Given a map $\hat{\chi} : \{1, \ldots, n'\} \to \{\ell, r, b\}$, form its $\ell$-$r$-replacement $\chi : \{1, \ldots, n\} \to \{\ell, r\}$ where $n = n' + |\hat{\chi}^{-1}(\{b\})|$. Let $f_{\hat{\chi}}(i) = i + |\hat{\chi}^{-1}(\{b\})\cap[1,i-1]|$ for each $1 \leq i \leq n'$ and define $BNC_{\text{ffb}}(\hat{\chi})$ to be the subset of $BNC(\chi)$ such that $\pi \in BNC_{\text{ffb}}(\hat{\chi})$ if and only if whenever $i \in V \in \pi$ and $i \in f_{\hat{\chi}}(\hat{\chi}^{-1}(\{b\}))$, it follows that $i+1 \in V$.
\end{definition}

\begin{example}
    As an example, consider the map $\hat{\chi} = r b \ell$. Then it's $\ell$-$r$-replacement is the map $\chi =  r \ell r \ell$. So $BNC_{\text{ffb}}(\hat{\chi})$ is the set of bi-non-crossing partitions $\pi$ that lie above $0_{BNC_{\text{ffb}}(\hat{\chi})} = \{\{1\},\{2,3\},\{4\}\}$. There are four of these, each is pictured below.

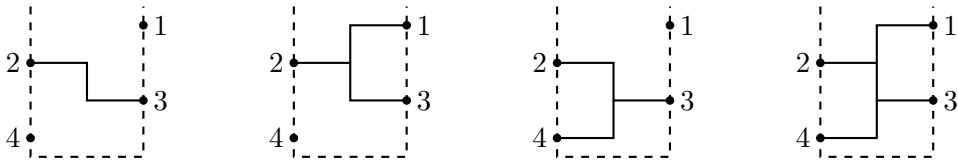
\begin{center}

    \begin{tikzpicture}[baseline]

	\draw[thick, dashed] (0,2) -- (0,0) -- (1.5, 0) -- (1.5,2);

	\draw[fill=black] (1.5, 1.75) circle (0.05);

	\node[right] at (1.5, 1.75) {$1$};

	\draw[fill=black] (0, 1.25) circle (0.05);

	\node[left] at (0, 1.25) {$2$};

	\draw[fill=black] (1.5, 0.75) circle (0.05);

	\node[right] at (1.5, 0.75) {$3$};

	\draw[fill=black] (0, 0.25) circle (0.05);

	\node[left] at (0, 0.25) {$4$};

        \draw[thick] (0, 1.25) -- (0.75, 1.25) -- (0.75, 0.75) -- (1.5,0.75);

	\draw[thick, dashed] (3.5,2) -- (3.5,0) -- (5, 0) -- (5,2);

	\draw[fill=black] (5, 1.75) circle (0.05);

	\node[right] at (5, 1.75) {$1$};

	\draw[fill=black] (3.5, 1.25) circle (0.05);

	\node[left] at (3.5, 1.25) {$2$};

	\draw[fill=black] (5, 0.75) circle (0.05);

	\node[right] at (5, 0.75) {$3$};

	\draw[fill=black] (3.5, 0.25) circle (0.05);

	\node[left] at (3.5, 0.25) {$4$};

        \draw[thick] (3.5, 1.25) -- (4.25, 1.25) -- (4.25, 0.75) -- (5,0.75);

        \draw[thick] (5, 1.75) -- (4.25, 1.75) -- (4.25, 1.25);

        \draw[thick, dashed] (7,2) -- (7,0) -- (8.5, 0) -- (8.5,2);

	\draw[fill=black] (8.5, 1.75) circle (0.05);

	\node[right] at (8.5, 1.75) {$1$};

	\draw[fill=black] (7, 1.25) circle (0.05);

	\node[left] at (7, 1.25) {$2$};

	\draw[fill=black] (8.5, 0.75) circle (0.05);

	\node[right] at (8.5, 0.75) {$3$};

	\draw[fill=black] (7, 0.25) circle (0.05);

	\node[left] at (7, 0.25) {$4$};

        \draw[thick] (7, 1.25) -- (7.75, 1.25) -- (7.75, 0.75) -- (8.5,0.75);

        \draw[thick] (7, 0.25) -- (7.75, 0.25) -- (7.75, 0.75);

        \draw[thick, dashed] (10.5,2) -- (10.5,0) -- (12, 0) -- (12,2);

	\draw[fill=black] (12, 1.75) circle (0.05);

	\node[right] at (12, 1.75) {$1$};

	\draw[fill=black] (10.5, 1.25) circle (0.05);

	\node[left] at (10.5, 1.25) {$2$};

	\draw[fill=black] (12, 0.75) circle (0.05);

	\node[right] at (12, 0.75) {$3$};

	\draw[fill=black] (10.5, 0.25) circle (0.05);

	\node[left] at (10.5, 0.25) {$4$};

        \draw[thick] (10.5, 1.25) -- (11.25, 1.25) -- (11.25, 0.75) -- (12,0.75);

        \draw[thick] (10.5, 0.25) -- (11.25, 0.25) -- (11.25, 0.75);

        \draw[thick] (12, 1.75) -- (11.25, 1.75) -- (11.25, 1.25);

	\end{tikzpicture}

\end{center}

These boolean pairs (i.e. pairs of numbers that replaced a `$b$' colouring) act as ``walls" so that we cannot connect nodes on opposite sides of a boolean pair without including the boolean pair in the block.
    
\end{example}

\begin{remark} \label{sublatticecumulantsremark}
    Let us notice a couple easy facts about our set of lattices. Firstly note that if $\hat{\chi}^{-1}(\{b\}) = \emptyset$, then its $\ell$-$r$-replacement $\chi = \hat{\chi}$, and in particular,
    \[ BNC_{\text{ffb}}(\hat{\chi}) = BNC(\chi).\]

    Furthermore, for any $\hat{\chi} : \{1, \ldots, n'\} \to \{ \ell, r, b\}$, if $\chi : \{1, \ldots, n\} \to \{ \ell, r\}$ is its $\ell$-$r$-replacement where $n = n' + |\hat{\chi}^{-1}(\{b\})|$ and $f_{\hat{\chi}}(i) = i + |\hat{\chi}^{-1}(\{b\}) \cap [1, i-1]|$ then $BNC_{\text{ffb}}(\hat{\chi})$ is an interval of the lattice $BNC(\chi)$ with maximum element $1_{BNC(\chi)}$ and minimum element $0_{BNC_{\text{ffb}}(\hat{\chi})}$ defined by
    \[ 0_{BNC_{\text{ffb}}(\hat{\chi})} = \{ \{i\}, \{j, j+1\} : i \in f_{\hat{\chi}}(\{1, \ldots, n'\} \setminus \hat{\chi}^{-1}(\{b\})), j \in f_{\hat{\chi}}(\hat{\chi}^{-1}(\{b\})\} \] 
    This means that $\mu_{BNC}$ can be restricted to $BNC_{\text{ffb}}$. That is, if $\pi, \sigma \in BNC_{\text{ffb}}(\hat{\chi})$ then 
    \[ \sum_{\substack{\tau \in BNC_{\text{ffb}}(\hat{\chi}) \\ \pi \leq \tau \leq \sigma}} \mu_{BNC}(\tau, \sigma) = \sum_{\substack{\tau \in BNC_{\text{ffb}}(\hat{\chi}) \\ \pi \leq \tau \leq \sigma}} \mu_{BNC}(\pi, \tau) = \begin{cases} 1 & \text{ if } \pi = \sigma \\ 0 &\text{ otherwise} \end{cases}.\]

    So the M\"obius function on $BNC(\chi)$ acts as a M\"obius function on $BNC_{\text{ffb}}(\hat{\chi})$ when restricted, giving us a M\"obius inversion therein.
\end{remark}

Now we aim to show that we can indeed use the bi-free cumulants as free-free-Boolean cumulants. To do so, we look at the moments if a bi-free ffb $B$-system, with our bi-non-crossing partitions above in mind.

\begin{theorem} \label{vanishingffbpartitions}
    Let $\{(A_k^\ell, A_k^r, C_k', D_k')\}$ be a bi-free ffb $B$-system in a $B$-$B$-non-commutative probability space $(\A,E_\A,\varepsilon)$. Fix
    $\hat{\upchi}:\{1,\ldots, n'\} \to \{\ell, r, b\}$. Let $\chi = \{1, \ldots, n\} \to \{\ell, r\}$ be the $\ell$-$r$-replacement of $\hat{\chi}$, where $n = n' + |\hat{\chi}^{-1}(\{b\})|$. For each $i \in \{1, \ldots, n'\}$ define $f_{\hat{\chi}}(i) = i + |\hat{\chi}^{-1}(b) \cap [1,i-1]|$. Let $\epsilon : \{1, \ldots, n\} \to K$ be such that $\epsilon(i) = \epsilon(i+1)$ whenever $i \in f_{\hat{\chi}}(\hat{\chi}^{-1}(\{b\}))$. Now let 
    \[ Z_i \in \begin{cases} \A_{\epsilon(i)}^{\chi(i)} &\text{ if } i \in f_{\hat{\chi}}(\{1, \ldots, n\} \setminus \hat{\chi}^{-1}(\{b\})) \\
    C_{\epsilon(i)}' &\text{ if } i \in f_{\hat{\chi}}(\hat{\chi}^{-1}(\{b\}) \\
    D_{\epsilon(i)}' &\text{ otherwise}\end{cases}.\]

    If $\pi \in BNC(\chi)$ and $\pi \leq \epsilon$, then
    \[ \E_\pi(Z_1, \ldots, Z_n ) =0 \]
    unless $\pi \in BNC_{\text{ffb}}(\hat{\chi})$.
\end{theorem}

\begin{proof}

    Our goal amounts to showing that if $\pi$ is a partition which splits up a boolean pair, then the corresponding $\pi$-moment must be zero. 

    Fix $\pi \in BNC(\chi)$. Consider the set $f_{\hat{\chi}}(\hat{\chi}^{-1}(\{b\})) = \{ i_1 < i_2 < \cdots < i_m \}$. We argue that whenever $i_j \in V \in \pi$, then either $i_j + 1 \in V$, or the corresponding $\pi$-moment is zero. We proceed inductively.

    For the base case we show that if $V \in \pi$ with $i_1 \in V$, then $i_1 + 1 \in V$, otherwise $\E_\pi(Z_1, \ldots, Z_n) = 0$. Indeed, suppose that there is some block $W \in \pi$ with $W \neq V$ such that $i_1 + 1 \in W$. We consider some cases. 

    \underline{Case 1: There is some $j > i_1$ such that $\chi(j) = r$ and $j \in V$.}
    If this is the case $\chi^{-1}(\{\ell\}) \cap [i_1, n] \cap W = \emptyset$. So since $\pi \in BNC(\chi)$, $W$ does not contain any $i_\alpha$ with $1 \leq \alpha \leq m$. 

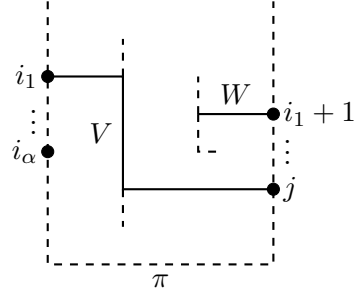
    \begin{center}

    \begin{tikzpicture}[baseline]

	\draw[thick, dashed] (0,3) -- (0,-0.5) -- (3, -0.5) -- (3,3);

	\draw[fill=black] (0, 2) circle (0.08);

	\node[left] at (0, 2) {$i_1$};

	\draw[fill=black] (3, 1.5) circle (0.08);

	\node[right] at (3, 1.5) {$i_1 + 1$};

	\draw[fill=black] (3, 0.5) circle (0.08);

        \draw[fill=black] (0, 1) circle (0.08);

	\node[right] at (3, 0.5) {$j$};

        \node[left] at (1, 1.25) {$V$};

        \draw[thick, dashed] (1, 0.5) -- (1, 0);

        \draw[thick, dashed] (1, 2) -- (1, 2.5);

        \draw[thick] (0, 2) -- (1, 2) -- (1, 0.5) -- (3,0.5);

        \draw[thick] (3, 1.5) -- (2, 1.5);

        \draw[thick, dashed] (2, 1.5) -- (2, 2);

        \draw[thick, dashed] (2, 1.5) -- (2, 1) -- (2.25,1);

        \node[above] at (2.5, 1.5) {$W$}; 

        \node[right] at (3, 1.1) {$\vdots$};

        \node[left] at (0, 1.5) {$\vdots$};

        \node[below] at (1.5, -0.5) {$\pi$};

        \node[left] at (0, 1) {$i_\alpha$};
	\end{tikzpicture}

\end{center}
    
    So the ordered product $\prod_{i \in W} Z_i$ 
    is an alternating product of elements from $\bA_{\epsilon(W)} = \text{alg}(A_{\epsilon(W)}^\ell, A_{\epsilon(W)}^r)$ and elements of $D_{\epsilon(W)}'$. In fact, since $A_{\epsilon(W)}^\ell$ and $A_{\epsilon(W)}^r$ are $B$-faces, we can add left and right multiplications by $B$ without changing this fact. In particular, for any $b_{1,1}, b_{1,2}, b_{2,1}, \ldots, b_{n,1},b_{n,2}, b_{n+1} \in B$, we have that the ordered product
    \[\left(\prod_{i \in W} L_{b_{i,1}}R_{b_{i,2}}Z_i\right) \cdot L_{b_{n+1}} \]
    is still an alternating product of elements from $\bA_{\epsilon(W)}$ and $D_{\epsilon(W)}'$.
    Thus by property (3) in the definition of bi-free ffb $B$-systems, we have  
    \[ \E_{\pi|_W}((L_{b_{1,1}}R_{b_{1,2}}Z_1, L_{b_{2,1}}R_{b_{2,2}}Z_2, \ldots, L_{b_{n,1}} R_{b_{n,2}}Z_nL_{b_{n+1}})|_W) = 0.\]
    Now by the reduction process in Definition \ref{bimultiplicative} (of bi-multiplicative functions), there will be a stage where we must evaluate an expectation of the form above, where the elements $b_{1,1}, \ldots, b_{2,n}, b_{n+1}$ act as potential placeholders for other expectations of blocks. So from above, it follows from bi-multiplicativity, that
    \[ \E_\pi(Z_1, \ldots, Z_n) = 0.\]

    \underline{Case 2: There are no elements $j > i_1$ with $\chi(j)=r$ and $j \in V$.}

    If this is the case, since $\pi \in BNC(\chi)$, $V$ will not contain any elements $i_\alpha + 1$ for any $1 \leq \alpha \leq m$. 

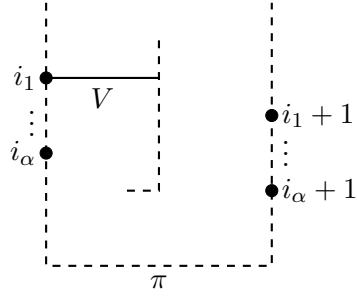
    \begin{center}

    \begin{tikzpicture}[baseline]

	\draw[thick, dashed] (0,3) -- (0,-0.5) -- (3, -0.5) -- (3,3);

	\draw[fill=black] (0, 2) circle (0.08);

	\node[left] at (0, 2) {$i_1$};

	\draw[fill=black] (3, 1.5) circle (0.08);

	\node[right] at (3, 1.5) {$i_1 + 1$};

	\draw[fill=black] (3, 0.5) circle (0.08);

        \draw[fill=black] (0, 1) circle (0.08);

	\node[right] at (3, 0.5) {$i_\alpha + 1$};

        \node[below] at (0.75, 2) {$V$};

        \draw[thick, dashed] (1.5, 2.5) -- (1.5, 0.5) -- (1, 0.5);

        \draw[thick] (0, 2) -- (1.5, 2);

        \node[right] at (3, 1.1) {$\vdots$};

        \node[left] at (0, 1.5) {$\vdots$};

        \node[below] at (1.5, -0.5) {$\pi$};

        \node[left] at (0, 1) {$i_\alpha$};
	\end{tikzpicture}

\end{center}

    So similar to above, the ordered product of elements indexed by $V$ is an alternating product of elements from $\bA_{\epsilon(V)}$ and elements of $C_{\epsilon(V)}'$. The same is true if we include left and right multiplications by elements of $B$. That is, for any $b_{1,1}, b_{1,2}, b_{2,1}, \ldots, b_{n,1},b_{n,2}, b_{n+1} \in B$, we have that the ordered product
    \[ \left( \prod_{i \in V} L_{b_{i,1}} R_{b_{i,2}} Z_i \right) \cdot L_{b_{n+1}}\]
    is an alternating product of elements from $\bA_{\epsilon(V)}$ and elements of $C_{\epsilon(V)}'$. Thus by property (2) of the definition of bi-free ffb $B$-systems, we have 
    \[ \E_{\pi|_V}((L_{b_{1,1}}R_{b_{1,2}}Z_1, L_{b_{2,1}}R_{b_{2,2}}Z_2, \ldots, L_{b_{n,1}} R_{b_{n,2}}Z_nL_{b_{n+1}})|_V) = 0.\]
    Bi-multiplicativity then implies 
    \[ \E_\pi(Z_1, \ldots, Z_n) = 0.\]

    $\ $

    Now suppose that for $k \geq 1$, we have that whenever $i_j \in V \in \pi$, it follows that $i_j + 1 \in V$, for all $1 \leq j \leq k$. Let $V \in \pi$ such that $i_{k+1} \in V$. Suppose that $i_{k+1} +1 \in W \neq V$. We consider cases similarly as above. 

    Firstly, if there is some $\alpha > i_{k+1}$ such that $\chi(\a) = r$ and $\a \in V$, then $W \cap \chi^{-1}(\{\ell\}) \cap [i_{k+1}+1,n] = \emptyset$, since $\pi \in BNC(\chi)$. By the hypothesis above we have that for any $1 \leq j \leq k$, $i_j \in W$ if and only if $i_j + 1 \in W$. This implies, since $\pi$ is bi-non-crossing, that if such a $j$ exists, then in fact $i_k, i_k + 1 \in W$. If this is the case, then letting 
    \[ A_1 = \prod_{i \in W \cap [i_k + 2, i_{k+1}]}Z_i \in \bA_{\epsilon(W)}  \quad \text{ and } \quad A_2 = \prod_{i \in W \cap [i_{k+1}+2, n]} Z_i \]
    where the products are ordered. So property (1) of the definition of bi-free ffb $B$-system says that $Z_{i_k+1}A_1Z_{i_{k+1}+1} = 0$, yielding 
    \[ \prod_{i \in W} Z_i = \left( \prod_{i \in W \cap [1, i_k]} Z_i \right) \cdot Z_{i_k+1} A_1 Z_{i_{k+1}+1}A_2 = 0, \]
    hence the corresponding $\pi|_W$ moment is zero. If instead $i_k \notin W$, then property (3) guarantees that 
    \[ \E_{\pi|_W}((Z_1, \ldots, Z_n)|_W) = 0.\]
    As above, we can upgrade this with left and right multiplications by elements of $B$ to make use of bi-multiplicativity granting 
    \[ \E_\pi(Z_1, \ldots, Z_n) = 0. \]

    Otherwise, if there is no such $\a > i_{k+1}$ such that $\chi(\a) = r$ and $\a \in V$, then $V \cap \chi^{-1}(\{r\}) \cap [i_{k+1}, n] = \emptyset$. Let $j > k$ be the smallest natural number such that $i_j \in V$, if it exists. Define the ordered products
    \[ A_1 = \prod_{i \in V \cap [i_{k+1}+1, i_j -1]} Z_i \in \bA_{\epsilon(V)} \quad \text{ and } \quad A_2 = \prod_{i \in V \cap [i_j +1, n]} Z_i, \]
    then property (1) again says that $Z_{i_{k+1}}A_1Z_{i_j} = 0$, so that the ordered product
    \[ \prod_{i \in V}Z_i = \left( \prod_{i \in V \cap [1, i_{k+1}-1]} Z_i\right)Z_{i_{k+1}} A_1 Z_{i_j}A_2 = 0,\]
    so the corresponding $\pi|_V$ moment is zero. If no such $j$ exists, then by hypothesis, for any $1 \leq j \leq k$,  $i_j \in V$ if and only if $i_j + 1 \in V$. This implies that 
    \[ \E_{\pi|_V}((Z_1, \ldots, Z_n)|_V) = 0\]
    as this is a moment of the form of property (2) in the definition of bi-free ffb $B$-system. We may upgrade this observation with left and right multiplications by elements of $B$ and use bi-multiplicativity to get that
    \[ \E_\pi(Z_1, \ldots, Z_n) = 0. \qedhere \]
\end{proof}

This result grants us use of the bi-free cumulants to characterize the free-free-Boolean independence as follows. Note the analog between this result and Theorem \ref{combinatorialbifree}.

\begin{theorem}
    Let $\{(A_k^\ell, A_k^r, C_k', D_k')\}_{k \in K}$ be a bi-free ffb B-system in a $B$-$B$-non-commutative probability space $(\A, E_\A, \varepsilon)$. Let $\hat{\chi} : \{1, \ldots, n'\} \to \{ \ell, r, b\}$, and define $f_{\hat{\chi}}(i) = i + |\hat{\chi}^{-1}(\{b\})\cap[0, i-1]|$. Let $\chi : \{ 1, \ldots, n\}, \to \{\ell, r\}$ be the $\ell$-$r$-replacement of $\hat{\chi}$, where $n = n'  + |\hat{\chi}^{-1}(\{b\})|$. Let $\epsilon : \{ 1, \ldots, n\} \to K$ be such that $\epsilon(i) = \epsilon(i+1)$ for all $i \in f(\hat{\chi}^{-1}(\{b\}))$. For each $i \in \{1, \ldots, n\}$ if 
    \[ Z_i \in \begin{cases} A_{\epsilon(i)}^{\chi(i)} & \text{ if } i \in f_{\hat{\chi}}(\{1, \ldots, n'\} \cap \hat{\chi}^{-1}(\{b\})) \\
    C_k' &\text{ if } i \in f_{\hat{\chi}}(\hat{\chi}^{-1}(\{b\})) \\
    D_k' &\text{ otherwise} \end{cases},\]
    then 
    \[ E(Z_1 Z_2 \cdots Z_n) = \sum_{\pi \in BNC_{\text{ffb}}(\hat{\chi})}\kappa_\pi(Z_1, \ldots, Z_n).\]
    Equivalently,
    \[ \kappa_{1_\chi}(Z_1, \ldots, Z_n) = \sum_{\pi \in BNC_{\text{ffb}}(\hat{\chi})} \E_\pi(Z_1, \ldots, Z_n)\mu_{BNC}(\pi, 1_{\chi}).\]
    Lastly, $\kappa_{1_\chi}(Z_1, \ldots, Z_n) = 0$ unless $\epsilon$ is constant.
\end{theorem}

\begin{proof}
    Using the fact that the pairs $\{(\text{alg}(A_k^\ell,C_k'),\text{alg}(A_k^r,D_k'))\}_{k \in K}$ are bi-free and Theorem \ref{combinatorialbifree},
    \begin{align*}
        E_\A(Z_1 Z_2 \cdots Z_n) &= \sum_{\pi \in BNC(\chi)} \sum_{ \substack{\sigma \in BNC(\chi) \\\pi \leq \sigma \leq \epsilon}} \mu_{BNC}(\pi, \sigma) \E_\pi(Z_1, Z_2, \ldots, Z_n) \\
        &= \sum_{\pi \in BNC_{\text{ffb}}(\chi)} \sum_{\substack{\sigma \in BNC(\chi)\\\pi \leq \sigma \leq \epsilon}} \mu_{BNC}(\pi, \sigma) \E_\pi(Z_1, Z_2, \ldots, Z_n) &\text{(Theorem \ref{vanishingffbpartitions})} \\
        &= \sum_{\pi \in BNC_{\text{ffb}}(\chi)} \sum_{\substack{\sigma \in BNC_{\text{ffb}}(\chi)\\\pi \leq \sigma \leq \epsilon}} \mu_{BNC}(\pi, \sigma) \E_\pi(Z_1, Z_2, \ldots, Z_n) &\text{(Remark \ref{sublatticecumulantsremark})} \\
        &= \sum_{\substack{\sigma \in BNC_{\text{ffb}}(\chi) \\ \sigma \leq \epsilon}} \sum_{\substack{\pi \in BNC_{\text{ffb}}(\chi)\\\pi \leq \sigma}} \mu_{BNC}(\pi, \sigma) \E_\pi(Z_1, Z_2, \ldots, Z_n) \\
        &=\sum_{\substack{\sigma \in BNC_{\text{ffb}}(\chi) \\ \sigma \leq \epsilon}} \sum_{\substack{\pi \in BNC(\chi)\\\pi \leq \sigma}} \mu_{BNC}(\pi, \sigma) \E_\pi(Z_1, Z_2, \ldots, Z_n) &\text{(Theorem \ref{vanishingffbpartitions})}\\
        &= \sum_{\sigma \in BNC_{\text{ffb}}(\chi)}\kappa_{\sigma}(Z_1, Z_2, \ldots, Z_n).
    \end{align*}

    The last claim follows again from the fact that the pairs $\{(\text{alg}(A_k^\ell, C_k'), \text{alg}(A_k^r,D_k'))\}_{k \in K}$ are bi-free.
\end{proof}

So these special bi-free cumulants will act as free-free-Boolean cumulants once we can show that our bi-free ffb $B$-systems give rise to free-free-Boolean independence, which we will do next.

\subsection{Independence of our abstract structures}

We saw in section 5.2 a construction that embeds any free-free-Boolean family into a bi-free ffb $B$-system, but there was a lingering question about whether the joint distribution was preserved. With the combinatorics tools that we have just developed, we can now show that bi-free ffb $B$-systems do in fact build free-free-Boolean independent families with amalgamation over $B$, thus ensuring that our embedding preserves joint distribution.

\begin{theorem} \label{abstractstructuresareffb}
    If $\{ (A_k^\ell, A_k^r, C_k', D_k')\}_{k \in K}$ is a bi-free ffb $B$-system in a $B$-$B$-non-commutative probability space $(\A, E_\A, \varepsilon)$ then the family $\Gamma = \{(A_k^\ell, A_k^r, \text{alg}(C_k'D_k'))\}_{k \in K}$ is free-free-Boolean independent with amalgamation over $B$.
\end{theorem}

\begin{proof}
    This argument compares an arbitrary moment in the family $\Gamma$ with a corresponding free-free-Boolean represented moment as in the definition of free-free-Boolean independence. To do so, we set up several notations. 

    For each $k \in K$ let $(\X_k, \osc{\X}_k, p_k)$ be a copy of the $B$-$B$-bimodule with specified $B$-projection that appears in Theorem \ref{faithfulbimodule} and let $\ell_k : A_k^\ell \to \L_\ell(\X_k)$, $r_k : A_k^r \to \L_r(\X_k)$, and $m_k : A_k^b := \text{alg}(C_k'D_k') \to \L(\X_k)$ all be defined as their corresponding restrictions of the map $\theta$ as defined in Theorem \ref{faithfulbimodule}. Let $(\X, \osc{\X}, p)$ be the reduced free product of the family $\{(\X_k, \osc{\X}_k, p_k)\}_{k \in K}$, let $\lambda_k$ and $\rho_k$ denote the corresponding left and right regular representations respectively.
    
    Henceforth, let $n \in \bN$, $\hat{\chi} : \{1, \ldots, n\} \to \{\ell, r,b\}$, $\hat{\epsilon} : \{1, \ldots, n\} \to K$, and $Z_1, \ldots, Z_n \in \A$ such that $Z_i \in A_{\hat{\epsilon}(i)}^{\hat{\chi}(i)}$. For those $i$ such that $\hat{\chi}(i) = b$, we restrict, without loss of generality, the $Z_i$ to be of the form $Z_i = S_{C,i}S_{D,i}$ where $S_{C,i}\in C_{\hat{\epsilon}(i)}'$ and $S_{D,i} \in D_{\hat{\epsilon}(i)}'$. 

    Now let $m = n + |\hat{\chi}^{-1}(\{b\})|$ and $f : \{1 ,\ldots, n\} \to \{1, \ldots, m\}$ be defined by 
    \[ f(i) = i + |\hat{\chi}^{-1}(\{b\}) \cap [1, i-1]|.\]
    
    We want to split up the Boolean factors so that we can make use of bi-freeness and our LR-diagram calculus. So with this idea, define for each $j \in \{1, \ldots, m\}$ 
    \[ T_j := \begin{cases} Z_{f^{-1}(j)} & \text{ if } j \in f(\hat{\chi}^{-1}\{\ell, r\}) \\ 
    S_{C,i} &\text{ if } i \in \hat{\chi}^{-1}(\{b\}) \text{ and } f(i) = j \\
    S_{D,i} & \text{ if } i \in \hat{\chi}^{-1}(\{b\}) \text{ and } f(i) + 1 = j 
    \end{cases}.\]

    For convenience, also let $\chi : \{1, \ldots, m\} \to \{\ell, r\}$ be the $\ell$-$r$-replacement of $\hat{\chi}$, and let $\epsilon : \{1,\ldots, m\} \to K$ be defined from $\hat{\epsilon}$ by $\epsilon(f(i)) = \hat{\epsilon}(i)$, and if $j \in \{1, \ldots, m\} \setminus f(\{1,\ldots, n\})$, let $\epsilon(j) = \epsilon(j-1)$ so that $\epsilon(f(i)) = \epsilon(f(i) + 1)$ whenever $i \in \hat{\chi}^{-1}(\{b\})$.
    
    Furthermore, for each $1 \leq i \leq m$ let
    \[ \mu_i(T_i) = \begin{cases} \lambda_{\epsilon(i)}(\ell_{\epsilon(i)}(T_i)) & \text{if } \chi(i) = \ell \\ \rho_{\epsilon(i)}(r_{\epsilon(i)}(T_i)) & \text{if } \chi(i) = r \end{cases}, \quad \text{ and } \quad  \mu_i'(T_i) = \begin{cases} P_{\epsilon(i)}\mu_i(T_i) & \text{if } i \in f(\hat{\chi}^{-1}(\{b\}) \\
    \mu_i(T_i) &\text{otherwise} \end{cases}, \] 
    and
    \[\tilde{\mu_i}(Z_i) = \begin{cases} \lambda_{\epsilon(i)}(\ell_{\epsilon(i)}(Z_i)) & \text{ if } \hat{\chi}(i) = \ell \\ \rho_{\epsilon(i)}(r_{\epsilon(i)}(Z_i)) & \text{ if } \hat{\chi}(i) = r \\ P_{\epsilon(i)}\lambda_{\epsilon(i)}(m_{\epsilon(i)}(Z_i))P_{\epsilon(i)} & \text{ if } \hat{\chi}(i) = b \end{cases}. \]

    With all of our setup in hand, we note that the $Z_i$ are chosen to correspond to an arbitrary moment of $\Gamma$ for which we would like to compare with the corresponding free-free-Boolean moment. That is, our proof will be complete if we can demonstrate that  
    \[ E_\A(Z_1 \cdots Z_n) = E_{\L(\X)}(\tilde{\mu}_1(Z_1) \cdots \tilde{\mu}_n(Z_n)). \]
    
    The elements $T_i$ will help us form a bridge between these two moments. To see how, first note that by definition of the $T_i$, we have $Z_1 \cdots Z_n = T_1 \cdots T_m$, so that
    \begin{equation} \label{ZtoT}
        E_\A(Z_1 \cdots Z_n) = E_\A(T_1\cdots T_m)
    \end{equation}
    Since the family $\{(\text{alg}(A_k^\ell, C_k'), \text{alg}(A_k^r, D_k'))\}_{k \in K}$ is bi-free over $B$, we have 
    \begin{equation} \label{EATtoELT}
        E_\A(T_1 \cdots T_m) = E_{\L(\X)}(\mu_1(T_1) \cdots \mu_m(T_m)) = p\mu_1(T_1)\cdots\mu_m(T_m)1_B.
    \end{equation}

    Next, for each $i \in \hat{\chi}^{-1}(\{b\})$, we have
    \begin{align*}
        \mu_{f(i)}'(T_{f(i)})\mu_{f(i)+1}'(T_{f(i)+1}) &= P_{\epsilon(f(i))}\mu_{f(i)}(T_{f(i)})\mu_{f(i)+1}(T_{f(i)+1}) \\ 
        &= P_{\epsilon(f(i))} \lambda_{\epsilon(f(i))}(\theta(S_{C,i}))\rho_{\epsilon(f(i)+1)}(\theta(S_{D,i})) \\
         &= P_{\epsilon(f(i))}\lambda_{\epsilon(f(i))}(\theta(S_{C,i}))\rho_{\epsilon(f(i))}(\theta(S_{D,i})) \\
         &= P_{\epsilon(f(i))}\lambda_{\epsilon(f(i))}(\theta(S_{C,i}))\lambda_{\epsilon(f(i))}(\theta(S_{D,i})) &\text{(Lemma \ref{boolprojfacts})}\\
         &= P_{\epsilon(f(i))}\lambda_{\epsilon(f(i))}(\theta(S_{C,i}S_{D,i})) \\
         &= P_{\epsilon(f(i))}\lambda_{\epsilon(f(i))}(m_i(Z_i))P_{\epsilon(f(i))} \\
         &=  P_{\hat{\epsilon}(i)}\lambda_{\hat\epsilon(i)}(m_i(Z_i))P_{\hat\epsilon(i)} \\
         &= \tilde{\mu}_i(Z_i),
    \end{align*}
    where we made use of Lemma \ref{boolprojfacts} to replace the $\rho_{\epsilon(f(i)+1)}$ with $\lambda_{\epsilon(f(i)+1)}$ (which is the same as $\lambda_{\epsilon(f(i))}$), and to commute those operators with the projections in the antepenultimate equality. From this we can see that, in fact,
    \begin{equation} \label{muprimeTstotildeZ}
        \mu_1'(T_1) \cdots \mu_m'(T_m) = \tilde{\mu}_1(Z_1)\cdots \tilde{\mu}_n(Z_n).
    \end{equation}

    For each $i \in \hat{\chi}^{-1}(\{b\})$, form the sets
    
    \[S_{f(i)} = LR^{\text{lat}}(\chi|_{\{f(i), \ldots, m\}},\epsilon|_{\{f(i),\ldots, m\}},\epsilon(f(i)))^c \]
    Then by Lemma \ref{keyprojLRlemma}, there is some $S \subseteq \bigcup_{i \in \hat{\chi}^{-1}(\{b\})} S_{f(i)}^\chi$ and constants $c_{D,S}$ such that
    \begin{equation} \label{LRdecompofTs}
        \mu_1(T_1)\cdots \mu_n(T_n)1_B = \mu_1'(T_1) \cdots \mu_n'(T_n)1_B + \sum_{D \in S}c_{D,S}E_D(\mu_1(T_1), \ldots, \mu_n(T_n)).
    \end{equation}

    Now from equation (\ref{muprimeTstotildeZ}), we only need to show that 
    \begin{equation} \label{Eofsuperfluousis0}
        E_{\L(\X)} \left(\sum_{D \in S}c_{D,S}E_D(\mu_1(T_1), \ldots, \mu_m(T_m))\right) = 0.
    \end{equation}
    It is immediate from Remark \ref{wherearetheLRpieces} that if $D \in LR^{\text{lat}}(\chi,\epsilon) \setminus LR_0^{\text{lat}}(\chi,\epsilon)$, then
    \[ E(c_{D,S}E_D(\mu_1(T_1), \ldots, \mu_m(T_m))) = 0,\]
    so we only need to consider those diagrams $D$ which correspond to partitions $\pi \in BNC(\chi)$ with $\pi \leq \epsilon$. To this end, let $D \in S = \bigcup_{j = 1}^\ell S_{i_j}^\chi \cap LR_0^{\text{lat}}(\chi,\epsilon)$ and let $\pi$ be its corresponding bi-non-crossing partition. This means that there is some $i \in \hat{\chi}^{-1}(\{b\})$ such that $D$ is a $\chi$-extension of a diagram in $S_{f(i)}$. But by definition of $S_{f(i)}$, $T_{f(i)}$ is not in the same block as $T_{f(i) + 1}$. So it remains true that any $\chi$-extension of such a diagram also leaves $T_{f(i)}$ and $T_{f(i)+1}$ disconnected. Thus, by Theorem \ref{vanishingffbpartitions} and bi-free independence, it follows that 
    \[  E_D(\mu_1(T_1), \ldots, \mu_m(T_m))= \E_\pi(T_1, \ldots, T_m) = 0.\]
    Combining these observations grants us equation (\ref{Eofsuperfluousis0}).

    We conclude from equations (\ref{ZtoT}), (\ref{EATtoELT}), (\ref{LRdecompofTs}), (\ref{muprimeTstotildeZ}), and (\ref{Eofsuperfluousis0}), that 
    \[ E_\A(Z_1 \cdots Z_n) = E_{\L(\X)}(\tilde{\mu}_1(Z_1) \cdots \tilde{\mu}_n(Z_n)),\]
    granting us, by the arbitrary choice of moment, that the family $\Gamma$ is free-free-Boolean independent over $B$. 
\end{proof}

So after connecting free-free-Boolean independent families with bi-free structures, in the form of our ffb $B$-systems, we see that the free-free-Boolean cumulants can be defined in terms of the bi-free cumulants. Thus, any appropriate combinatorial tools from bi-free probability helps elucidate the combinatorics of free-free-Boolean independence. This diminishes the need to develop the tools of another lattice, such as the so-called \textit{interval-bi-non-crossing partitions} of Liu in \cite{liu17}.

\section{A brief final word on other multi-algebra independences}

By considering the case when $B = \bC$, we see that the work above condenses down to a result about the free-free-Boolean independence (without the amalgmation over $B$) as defined in \cite{liu17}. Indeed, for example, if $B = \bC$, then $B$-$B$-non-commutative probability spaces become simply non-commutative probability spaces. While Liu's definition of the free-free-Boolean independence in \cite{liu17} may appear different, we note that by taking a suitable free product of the algebras that appear in that definition, we can match that with the definition provided here. Thus the non-amalgamated setting yields corresponding results for free.

In \cite{Liu19}, a notion of free-Boolean independence was discussed, and studied in the amalgamated setting in \cite{liu+zhong}. For the amalgamated free-Boolean independence, we can note simply from the work above that if $\{ (A^\ell_k, A_k^r, A_k^b)\}_{k \in K}$ is free-free-Boolean independent over $B$, then the family $\{ (A_k^\ell, A_k^b)\}_{k \in K}$ is free-Boolean independent over $B$. In the non-amalgamated setting, when $B = \bC$, the family $\{(A_k^\ell, A_k^b)\}_{k \in K}$ would be remain free-Boolean independent. By restricting our attention to the family $\{(A_k^\ell, A_k^b)\}_{k \in K}$ in the above work, we get corresponding results for free-Boolean independence (amalgamated over $B$ or not) for free. So in general, free-Boolean independence can be studied in bi-free probability as well.

Liu also suggests a more general description of similar multi-algebra independences in \cite{Liu19}. In particular, looking at the definition of free-free-Boolean independence (Definition \ref{ffbdef}), the family 
\[ \{(\lambda_k(\ell_k(A_k^\ell)), \rho_k(r_k(A_k^r)), P_k(\lambda_k(m_k(A_k^b)))P_k)\}_{k \in K} \]
has implicit identity projections flanking the left and right acting algebras. Alongside the Boolean projections and identity projections, there are also \textit{monotone projections} (see \cite{Liu19}). By mixing and matching these projections, we can define other types of multi-algebra independence. For example, if we replace the Boolean projections with the monotone projections, we get a notion of free-free-monotone independence. 

In \cite{Sko14} Skoufranis also constructed a model for a pair of monotone independent algebras. It seems likely that extending this model with similar constructions performed in this work would allow one to study free-free-monotone (and hence free-monotone) independence for pairs of tuples of algebras through the lens of bi-free probability as well.

\printbibliography

@article{Sko14,
author = "P. Skoufranis",
title = "Independences and Partial $R$-Transforms in Bi-Free Probability",
year = "2016",
journal = "Ann. Inst. H. Poincar\'e' Probab. Statis.",
pages = "1437--1473",
volume = "52"
}

@article{liu+zhong,
author = {W. Liu and P. Zhong},
title = {Free-Boolean independence with amalgamation},
journal = {Infinite Dimensional Analysis, Quantum Probability and Related Topics},
volume = {22},
number = {04},
year = {2019}
}

@article{liu17,
author = "W. Liu",
title = "Free-Free-Boolean Independence for Triples of Algebras",
year = "2017",
URL = "https://arxiv.org/abs/1801.03401"
}

@article{CNS15,
author = "I. Charlesworth, B. Nelson, and P. Skoufranis",
title = "Combinatorics of Bi-Freeness with Amalgamation",
year = "2015",
journal = "Communications in Mathematical Physics",
pages = "801--847",
volume = "338"
}

@article{Voi14,
author = "DV. Voiculescu",
title = "Free Probability for Pairs of Faces I",
year = "2014",
journal = "Communications in Mathematical Physics",
pages = "955--980",
volume = "332"
}

@article{Liu19,
author = "W. Liu",
title = {Free-Boolean independence for pairs of algebras},
journal = {Journal of Functional Analysis},
volume = {277},
number = {4},
pages = {994-1028},
year = {2019},
}

@article{Spe97,
author = "R. Speicher",
title = "On universal products",
journal = "Free probability theory, Fields Inst. Commun.",
volume = "12",
pages = {257-266},
year = "1997",
}

\end{document}